\documentclass[11pt,reqno]{amsart}    

\usepackage[margin=1in]{geometry}
\usepackage{amsfonts,amsmath,amssymb,enumerate}    
\usepackage{amsthm}    
\usepackage{mathrsfs}  
\usepackage{setspace}

\onehalfspace
\usepackage{tikz}    
\usetikzlibrary{arrows,matrix}
\usetikzlibrary{decorations.markings,shapes.geometric,shapes.misc}    

\numberwithin{equation}{section}    

\theoremstyle{plain}
\newtheorem{thm}{Theorem}[section]
\newtheorem{lem}[thm]{Lemma}
\newtheorem{prop}[thm]{Proposition}

\theoremstyle{definition}
\newtheorem{defn}[thm]{Definition}

\newtheorem{exmp}[thm]{Example}
\theoremstyle{remark}
\newtheorem{rem}[thm]{Remark}

\newtheorem*{rem*}{Remark}
\newtheorem*{ack}{Acknowledgements}

\newcommand{\bs}{\boldsymbol}

\newcommand{\be}{\begin{equation}}    
\newcommand{\ee}{\end{equation}}    
\newcommand{\beu}{\begin{equation*}}    
\newcommand{\eeu}{\end{equation*}}    
\newcommand{\bea}{\begin{eqnarray}}    
\newcommand{\eea}{\end{eqnarray}}    
\newcommand{\beaa}{\begin{eqnarray*}}    
\newcommand{\eeaa}{\end{eqnarray*}}    
\newcommand{\bmx}{\begin{pmatrix}}    
\newcommand{\emx}{\end{pmatrix}}

\newcommand{\del}{\partial}    
\newcommand{\g}{{\mathfrak g}}    
\newcommand{\h}{{\mathfrak h}}

\newcommand{\tpd}[1]{\textstyle{\frac{\del}{\del\smash{ #1}}}}
\newcommand{\pd}[1]{{\frac{\del}{\del\smash{ #1}}}}

\newcommand{\mf}{\mathfrak}
\newcommand{\mc}{\mathcal}

\newcommand{\alf}{{\textstyle{\frac{1}{2}}}}

\newcommand{\nn}{\nonumber}

\newcommand{\8}{{\infty}}

\newcommand{\Z}{{\mathbb Z}}

\newcommand{\C}{{\mathbb C}}

\renewcommand{\P}{{\mathcal P}}

\newcommand{\id}{{\mathrm{id}}}

\newcommand{\uq}{{U_q}}

\newcommand{\uqslt}{{\uq(\mf{sl}_2)}}    
\newcommand{\uqgh}{{\uq(\widehat\g)}}    
\newcommand{\uqlg}{{\uq(\mathcal L\g)}}
\newcommand{\uqg}{{\uq(\g)}}

\newcommand{\goi}[2]{=}    
\newcommand{\Hom}{\mathrm{Hom}}

\newcommand{\on}{.}    

\newcommand{\groth}[1]{{\mathrm{Rep}(#1)}}    
\newcommand{\Cx}{\mathbb C^\times}
\newcommand{\qnum}[1]{\left[ #1\right]_q}
\newcommand{\qbinom}[2]{\binom{#1}{#2}_q}
    
\newcommand{\btp}{\begin{tikzpicture}[baseline=0pt,scale=0.9,line width=0.25pt]}    
\newcommand{\etp}{\end{tikzpicture}}    

\newcommand{\Roff}{\color{black}}

\renewcommand{\L}{L}

\newcommand{\scr}{\mathscr}

\newcommand{\atp}[1]{}

\newcommand{\path}{\longrightarrow}

\newcommand{\cat}{\mathcal O}

\newcommand{\alg}{\mc A}
\newcommand{\algr}{{\overline{\mc A}}}
\newcommand{\cc}{\mc C}

\newcommand{\algf}[1]{\mathsf{#1}}
\newcommand{\E}{\algf E}
\renewcommand{\H}{\algf H}
\newcommand{\Er}{\overline\E}
\newcommand{\Hr}{\overline\H}
\newcommand{\HHr}{\overline\H}
\newcommand{\K}{\algf K}

\newcommand{\G}{\algf G}

\DeclareMathOperator{\End}{End}
\DeclareMathOperator{\Ob}{Ob}

\author{Charles Young}

\address{\vspace{-.15cm} 
School of Physics, Astronomy and Mathematics, University of Hertfordshire, College Lane, Hatfield AL10 9AB, UK.}  \email{charlesyoung@cantab.net}
\date{July 2012 \\ \emph{Email address:} \texttt{charlesyoung@cantab.net} }

\begin{document} 
\title{Quantum loop algebras and $\ell$-root operators}
\begin{abstract}
Let $\g$ be a simple Lie algebra over $\C$ and $q\in \Cx$ transcendental. We consider the category $\cc_\P$ of finite-dimensional representations of the quantum loop algebra $\uqlg$ in which the poles of all $\ell$-weights belong to specified finite sets $\P$. Given the data $(\g,q,\P)$, we define an algebra $\alg$ whose raising/lowering operators are constructed to act with definite $\ell$-weight (unlike those of $\uqlg$ itself). It is shown that there is a homomorphism $\uqlg\to\alg$ such that 
every representation $V$ in $\cc_P$ is the pull-back of a representation of $\alg$. 
\end{abstract}

\maketitle

\section{Introduction}
Quantum loop algebras and their finite-dimensional representations have been a topic of interest for two decades at least: for a recent review see \cite{CHbeyondKR}. Besides their original setting in integrable quantum- and statistical-mechanical models,
they appear in the contexts of 
algebraic geometry \cite{GV,Nakajima1,VV,Nakajima3}, combinatorics \cite{SchillingE6, LSS}, and cluster algebras \cite{HernandezLeclerc,NakajimaCluster}.


Given $\g$, a simple Lie algebra over $\C$, and $q\in \Cx$ transcendental, let $\uqlg$ be the corresponding quantum loop algebra and $\cc$ the category of its finite-dimensional representations.  
Let $(U^+,U^0,U^-)$ be the triangular decomposition of $\uqlg$ in Drinfeld's ``new'' realization \cite{Drinfeld,Beck}. The subalgebra $U^0\subset \uqlg$ is commutative, and any $V\in \Ob(\cc)$ can be decomposed into a direct sum of generalized eigenspaces of the generators of $U^0$.
The eigenvalues are known as \emph{$\ell$-weights}, and the  \emph{$q$-character} of $V$ is by definition the formal sum of its $\ell$-weights \cite{FR,Knight}. It is usually encoded as a Laurent polynomial $\chi_q(V)$ in formal variables $Y_{i,a}$, where $a\in \Cx$ and where $i$ runs over the set $I$ of nodes of the Dynkin diagram of $\g$. 
If one sends $Y_{i,a}\mapsto y_i:=e^{\omega_i}$ (where $\omega_i$ are the fundamental weights) one recovers the usual formal character $\chi(V)$ of $V$ regarded as a $\uqg$-module. In this sense $q$-characters refine the usual characters, and they have proven to be a powerful tool in understanding the structure of finite-dimensional representations \cite{FM, HernandezKR,HernandezMinAff, MY1,MY2}. 

\medskip

There is a fruitful analogy between the weight lattice $P$ of $\g$, and the so-called \emph{$\ell$-weight lattice} of $\uqlg$, which is defined to be the 
free abelian group $\P$ generated by the $Y_{i,a}$, $i\in I$, $a\in \Cx$.
\emph{Dominant} $\ell$-weights are the monomials in the $Y_{i,a}$, $i\in I$, $a\in \Cx$; they form a free monoid $\P^+\subset \P$. In the literature, dominant $\ell$-weights are usually called Drinfeld polynomials, and one of the first key results concerning $\cc$ was the classification of its irreducibles \cite{CPbook,CP94}:
the isomorphism classes of irreducible modules in $\cc$ are in bijection with the dominant $\ell$-weights. We write $\L(\bs\gamma)$ for the irreducible $\uqlg$-module with \emph{highest $\ell$-weight} $\bs\gamma\in \P^+$. 
Recall the corresponding classical result that the isomorphism classes of irreducible finite-dimensional $\g$-modules, and $\uqg$-modules, are in bijection with the dominant weights $P^+\subset P$; we write $V(\omega)$ for the irreducible $\uqg$-module with highest weight $\omega\in P^+$.
The analogy between the weight theory of $\g$ and its quantum-loop counterpart goes further: there are also counterparts $A_{i,a}\in \P$, $i\in I$, $a\in \Cx$ of the simple roots $\alpha_i\in P$, $i\in I$. 
We call them \emph{simple $\ell$-roots}.
Then, just as $\chi(V(\omega)) \in e^\omega \Z[e^{-\alpha_i}]_{i\in I}$, so also it is known that  $\chi_q(\L(\bs\gamma)) \in \bs\gamma\, \Z[A_{i,a}^{-1}]_{i\in I,a\in \Cx}$ \cite{FM}. 

However, at this stage the analogy with the usual weight theory breaks down, in the following important sense. 
Let $x^\pm_i$ and $k^{\pm 1}_i$, $i\in I$, be the usual Drinfeld-Jimbo generators of $\uqg$. We are accustomed to thinking of $x^\pm_i$ as step-operators between weight subspaces of a weight module $V$ of $\uqg$: if $v\in V$ is a weight vector of weight $\omega$ then $x^\pm_i\on v$ is again a weight vector, of weight $\omega\pm \alpha_i$. This follows, of course, from the defining relations 
\be k_i x_j^{\pm} = q^{\pm B_{ij}} x_j^\pm k_i,\nn\ee 
where $B_{ij}$ is the symmetrized Cartan matrix.
It is natural to ask for something similar in $\cc$: namely, to each simple $\ell$-root $A_{i,a}$ one would like to associate an operator ``$x_{A_{i,a}}^\pm$'' in $U^\pm$  such that for any $V\in \Ob(\cc)$ and any $v\in V$ of $\ell$-weight $\bs\gamma$,  ``$x_{A_{i,a}}^\pm\on v$'' has $\ell$-weight $\bs\gamma A_{i,a}^{\pm 1}$. But this is apparently too much to ask: the relevant commutation relation is
\be\phi^+_i(u)\,  x^\pm_j(v) =  \, \frac{q^{\pm B_{ij}}- uv}{1- q^{\pm B_{ij}}uv} \, x^\pm_j(v) \, \phi^+_i(u),\nn\ee
and $(A_{j,a}^{\pm 1})_i(u) :=  \frac{q^{\pm B_{ij}}- ua}{1- q^{\pm B_{ij}}ua}$ (details are recalled in \S2). Naively therefore, one can only obtain ``$x_{A_{i,a}}^\pm$'' by ``evaluating at $z=a$'' the formal series of generators $x_{i}^\pm(z):= \sum_{r\in \Z} z^{-r} x_{i,r}^\pm$. Such an infinite sum $x_{i}^\pm(a)= \sum_{r\in \Z} a^{-r} x_{i,r}^\pm$ is ill-defined in a sense that is not merely technical: its would-be matrix representatives have singular entries (c.f. equation (\ref{formofxpm}) below).

The lack of such operators ``$x_{A_{i,a}}^\pm$'' in $\uqlg$ is a problem. On the one hand  it makes it hard to use algebraic techniques to prove statements about $q$-characters: for example, the combinatorial Frenkel-Mukhin algorithm \cite{FM} is believed to yield the correct $q$-character for a much larger class of representations than the class for which one can currently prove that it does so. 
At the same time it means that the combinatorial structure the $q$-character -- which is in some sense elegant and sparse, certainly when compared to that of the usual  formal character -- can in practice be hard to lift to the level of representation theory. Ideally, one would like to have a general procedure to go from the $q$-character of a representation to an explicit $\ell$-weight basis, and for that one needs raising/lowering operators adapted to the $\ell$-weight decomposition.

\medskip

In the present work, it is argued that this obstruction 
can be circumvented by working not with $\uqlg$ directly, but rather with a new algebra $\alg$ whose raising/lowering operators act with definite $\ell$-weight by construction. $\alg$ will be defined in such a way that there is a homomorphism of algebras $\uqlg\to \alg$, allowing $\alg$-modules to be pulled back to recover $\uqlg$-modules.

The key observation (Propositions \ref{stepprop} and \ref{idea})  that makes such a construction possible  is the following.
For any finite-dimensional representation, say $V\in \Ob(\cc_\P)$, the matrix representatives $\rho(x_{i}^\pm(z))\in \End(V)[[z,z^{-1}]]$ take a very specific form. Namely,
\be \rho(x_{i}^\pm(z)) = \sum_{a\in \P_i}\sum_{m=0}^M E^\pm_{i,a,m} \frac{a^m}{m!} \left(\pd a\right)^m \delta\left(\frac a z\right),\label{formofxpm}\ee 
for certain maps $E^\pm_{i,a,m}\in \End(V)$. Here the $\P_i$ are finite collections of points in $\Cx$, and the upper limit $M$ is related to the maximal dimension of the $\ell$-weight spaces. Both depend on the choice of representation $V$. The $\delta$-function $\delta(a/z)$ is the formal distribution $\sum_{r\in \Z} (a/z)^r$, and the important point is that the dependence on the mode number $r$ of $x_{i,r}^\pm$ is solely through these formal $\delta$-functions and their derivatives (which, intuitively speaking, have ``support'' at the points $a\in \P_i$). 
It is then natural to ask what algebraic relations are obeyed by the maps $E^\pm_{i,a,m}$. The interesting relations are the ones that are independent of $V$, and, by   abstracting these, we arrive at the definition (\S\ref{algdef}) of $\alg$.

We shall work in subcategories $\cc_\P$ of $\cc$ (see Definition \ref{cpdef}) whose objects have $q$-characters lying in $\Z[Y_{i,a},Y_{i,aq_i^2}^{-1}]_{i\in I,a\in \P_i}$, where $(\P_i)_{i\in I}$ are arbitrary finite subsets of $\Cx$. Given any finite collection of finite-dimensional representations there always exists a choice of $\P$ such that they all belong to $\cc_\P$; to this extent, to study finite-dimensional representations it suffices to study such subcategories. Our definition of $\alg$ depends on $\P$, as well as on $\g$ and $q$.
(Particular categories of this form played a prominent role recently in \cite{HernandezLeclerc}.) 
The main result of the paper, Theorem \ref{mainthm}, is that a homomorphism $\uqlg\to \alg$ exists and that every representation in $\cc_\P$ arises as a pull-back of a representation of $\alg$. 

The $\E^\pm_{i,a,m}$, now to be thought of as abstract generators of $\alg$, are the desired operators of definite $\ell$-weight $A_{i,a}^{\pm 1}$. Meanwhile the ``Cartan'' generators $(\phi_{i,\pm r}^{\pm})_{i\in I,r\in \Z}$ of $U^0$ are mapped, in $\alg$, to generators $\H_{i,a,m}$, whose eigenvalues encode partial fraction decompositions of rational $\ell$-weights.  
The appearance of the extra label $m\in \Z_{\geq 0}$ is closely related to  a second sense in which the analogy between $\ell$-weights and the usual weight theory breaks down: 
whereas the $(k_i)_{i \in I}$ can be simultaneously diagonalized, the $(\phi_{i,\pm r}^{\pm})_{i\in I,r\in \Z}$ cannot; the $\ell$-weight encodes only their generalized eigenvalues with multiplicities. In highest weight representations, further information about their Jordan chains can now be read off from the commutation relations between $\H_{i,a,m}$'s and $\E^-_{i,a,m}$'s, cf. the examples in \S\ref{sl2sec}.

\medskip
Various natural questions about $\alg$ immediately arise and would be interesting to investigate. Here we shall merely note some of them. 
\begin{itemize}
\item One would like to construct a basis for $\alg$. Let $\alg^\pm$ and $\alg^0$ be the subalgebras of $\alg$ generated by the $\E^\pm_{i,a,m}$ and $\H_{i,a,m}$ respectively. In the present paper we establish that $\alg= \alg^-\cdot \alg^0 \cdot \alg^+$ (Proposition \ref{triangprop}). Because, for simplicity, we work with finite sets $\P_i$, the stronger statement $\alg\cong_\C \alg^-\otimes \alg^0 \otimes \alg^+$ is actually false (Proposition \ref{nottriang}), but were one to work instead with sets such that $q^{B_{ij}} \P_i = \P_i$ for all $i,j\in I$, then it is plausible that one would have such a triangular decomposition. Cf. Remark \ref{blockrem}. When $\g\neq \mf{sl}_2$, the form of the Serre relations, cf. Remark \ref{serrerem},  makes even the construction of bases of $\alg^\pm$ an interesting problem. (A basis of $\alg^\pm$ when $\g=\mf{sl}_2$ is given in Proposition \ref{sl2basisprop}.)
\item  It is well known that the untwisted quantum affine algebra $\uqgh$, and hence also $\uqlg$, has,  in addition to the standard Drinfeld-Jimbo coalgebra structure, also a ``Drinfeld current coproduct''. The latter is not a coproduct in a strict sense because it contains (in e.g. $\Delta x_{i}^+(z) = 1\otimes x_i^+(z) + x_i^+(z) \otimes \phi_i^-(1/z)$) formal sums that are ill-defined. See e.g. \cite{HernandezFusion2}. In view of the role of formal $\delta$-functions in the homomorphism $\uqlg\to \alg$, outlined above, it is reasonable to hope that in $\alg$ an analog of this coproduct can be made well-defined.  
\item For simplicity we assume $\g$ is simple but it is known \cite{Jing,Nakajima1,HernandezFusion} that the quantum affinization of $\uqg$ can be defined whenever $\g$ has symmetrizable Cartan matrix (so in particular when $\g$ is an affine Lie algebra). It would be interesting to establish whether the approach of the present paper goes through in that case too.
\item In simply-laced cases, there is a homomorphism from $\uqlg$ to the Grothendieck ring of the category of equivariant coherent sheaves on a certain Steinberg-type variety (endowed with the convolution product) \cite{Nakajima1}. It would be interesting to understand whether this homomorphism factors through (some generalization of) $\alg$. More tentatively, one can hope that $\alg$ allows for an algebraic approach to results that to date rely on geometrical input: notably the algorithm of \cite{Nakajima3} which, in the spirit of the Kazhdan-Lusztig conjecture, gives in principle the $q$-character of every irreducible of $\cc$.
\end{itemize}

\smallskip

This paper is structured as follows. Background results about $\uqlg$ and its finite-dimensional representations are recalled in \S\ref{qcharsec}. Then the first key result, namely the motivating observation -- (\ref{formofxpm}), above --  concerning the action of the raising/lowering operators on $\ell$-weight modules is proved in \S\ref{msec}, which also gives the definition of the categories $\cc_\P$. The algebra $\alg$ itself and main theorem, Theorem \ref{mainthm}, are given in section \S\ref{mainsec}. The proof of Theorem \ref{mainthm} is  in \S\ref{proofsec}. Finally in \S\ref{appssec} we note some first examples and properties of the algebra $\alg$; in particular that it appears to have a natural ``rational limit'', in the spirit of Yangians.

\medskip

\begin{ack} 
The author is very grateful to Evgeny Mukhin and Robin Zegers for useful discussions.
\end{ack}

\section{Background: Quantum Loop Algebras and $\ell$-weights}\label{qcharsec}
\subsection{Cartan data}
Let $\g$ be a simple Lie algebra over $\C$ and $\h$ a Cartan subalgebra of $\g$. We identify $\h$ and $\h^*$ by means of the invariant inner product $\left<\cdot,\cdot\right>$  on $\g$ normalized such that the square length of the maximal root equals 2. 
With $I$ a set of labels of the nodes of the Dynkin diagram of $\g$, let $\{\alpha_i\}_{i\in I}$ be a set of simple roots, and 
$\{\omega_i\}_{i\in   I}$, 
the corresponding set of
fundamental weights.
Let $C=(C_{ij})_{i,j\in I}$ be the Cartan matrix. We have 
\be\nn 2 \left< \alpha_i, \alpha_j\right> = C_{ij} \left<   \alpha_i,\alpha_i\right>,\quad 2 \left< \alpha_i, \omega_j\right> = \delta_{ij}\left<\alpha_i,\alpha_i\right>
.\ee 
Let $r^\vee$ be the maximal number of edges connecting two vertices of the Dynkin diagram of $\g$. Thus $r^\vee=1$ if $\g$ is of types A, D or E, $r^\vee = 2$ for types B, C and F and $r^\vee=3$ for $\mathrm G_2$.  Let $r_i= \alf r^\vee \left<\alpha_i,\alpha_i\right>$. The numbers $(r_i)_{i\in   I}$ are relatively prime integers. We set \be\nn D:= \mathrm{diag}(r_1,\dots,r_N),\qquad B := DC;\ee the latter is the symmetrized Cartan matrix, $B_{ij} = r^\vee \left<\alpha_i,\alpha_j\right>$. 
Let $Q$ (resp. $Q^+$) and $P$ (resp. $P^+$) denote the $\Z$-span (resp. $\Z_{\geq 0}$-span) of the simple roots and fundamental weights respectively. 
 


Fix a transcendental $q \in \Cx$. For each $i\in I$ let \be q_i:= q^{r_i}.\nn\ee 
Define the $q$-numbers, $q$-factorial and $q$-binomial: \be\nn \qnum n := \frac{q^n-q^{-n}}{q-q^{-1}},\quad \qnum n ! := \qnum n \qnum{n-1} \dots \qnum 1,\quad \qbinom n m := \frac{\qnum n !}{\qnum{n-m} ! \qnum m !}.\ee
We use also $\binom n m := \frac{n!}{(n-m)!m!}$ for the usual $(q=1)$ binomial coefficients.

\subsection{Quantum loop algebras}\label{uqlgdef}
The \emph{quantum loop algebra} $\uqlg$ is the associative unital algebra over $\C$ generated by 
\be (x^\pm_{i,n})_{i \in I, n \in \mathbb Z}, \quad\quad (k_i^{\pm 1})_{i \in I}, 
\quad\quad  (h_{i, r})_{i \in I , r \in \mathbb Z_{\neq 0}},\nn\ee
subject to the following relations. 
We arrange the generators into formal series     
\be x^\pm_i(u) := \sum_{n \in \mathbb Z} x^{\pm}_{i,n} u^{-n} \in \uqlg[[u,u^{-1}]]\nn ,\ee    
\be \phi_i^\pm (u) = \sum_{n=0}^\infty  \phi_{i,\pm n}^\pm u^{\pm n} 
:=  k_i^{\pm 1} \exp \left(\pm (q - q^{-1}) \sum_{m=1}^\infty h_{i,\pm m} u^{\pm m} \right )
\in \uqlg[[u^{\pm 1}]] \, , \label{phiu} \ee    
and set    
\be \delta(u) := \sum_{n \in \mathbb Z} u^n \in \C[[u,u^{-1}]].\nn\ee    
The defining relations of $\uqlg$ are then \cite{Drinfeld} $k_ik_i^{-1} = 1$ and
\be    
\left [\phi^\pm_i(u) , \phi^\pm_j (v) \right ] = 0,\quad \left [ \phi^\pm_i (u) , \phi^\mp_j(v) \right ] = 0, \label{phiphirel}\ee    
\be \left [x^+_i(u), x^-_j(v) \right ] = \delta_{ij} \delta\left(\frac u v\right) \frac{ \phi^+_i(1/u) -  \phi^-_i(1/u)}{q_i-q_i^{-1}}, \label{x+x-}\ee    
\bea    
(1- q^{\pm B_{ij}}uv)\phi^+_i(u)\,  x^\pm_j(v) &=&  \, (q^{\pm B_{ij}}- uv) \, x^\pm_j(v) \, \phi^+_i(u) , \label{phiprel}\\    
(1/(uv)- q^{\pm B_{ij}})\phi^-_i(u)\,  x^\pm_j(v) &=&  \, (q^{\pm B_{ij}}/(uv)-1) \, x^\pm_j(v) \, \phi^-_i(u), \label{phimrel}\\    
 \left ( u-q^{\pm B_{ij}}v \right )x^\pm_i(u) \, x^\pm_j(v) &=&     
\left (q^{\pm B_{ij}}u-v \right)\, x^\pm_j(v) \, x^\pm_i(u), \label{x+x+1}\eea    
together with Serre relations
\be\label{Serre} \sum_{\pi\in\Sigma_s}\sum_{r=0}^s(-1)^r
{\binom s r}_{\!q_i} 
x^\pm_{i}(w_{\pi(1)})\ldots
  x^\pm_{i}(w_{\pi(r)})  x^\pm_{j}(z) x^\pm_{i}(w_{\pi(r+1)})\ldots x^\pm_{i}(w_{\pi(s)}) =0,\ee
for all $i\neq j$, where $s=1-C_{ij}$ and $\Sigma_s$ is the symmetric group on $s$ letters.

Relations (\ref{phiprel}--\ref{phimrel}) are often written $[h_{i,n} , x_{j,m}^{\pm}] = \pm \frac{1}{n} [n B_{ij}]_qx_{j,n+m}^{\pm}$; see e.g. \cite{HernandezFusion} for a proof that they are equivalent.

In the present work it is useful have a slightly different presentation of $\uqlg$.
Let us define
\be \Phi_i(u) = \sum_{k\in \Z} \Phi_{i,k} u^k := \frac{\phi_i^+(u) - \phi^-(u)}{q_i-q_i^{-1}} \in \uqlg[[u,u^{-1}]].\label{Phidef}\ee
Here, as already in relation (\ref{x+x-}), the right-hand side is to be interpreted by extending $\phi^\pm_i(u)$ from a series in $\uqlg[[u^{\pm 1}]]$ to one in $\uqlg[[u,u^{-1}]]$ by setting $\phi^\pm_{i,\mp k}=0$ for all $k\in \Z_{\geq 1}$. So (\ref{Phidef}) is equivalent to
\be \Phi_{i,0} := \frac{\phi^+_{i,0} - \phi^-_{i,0}}{q_i-q_i^{-1}} = \frac{k_i - k_i^{-1}}{q_i-q_i^{-1}}, \quad \Phi_{i,\pm k} := \pm\frac{\phi^\pm_{i,\pm k}}{q_i-q_i^{-1}},\quad k\in \Z_{\geq 1}.\label{Phidef2}\ee
\begin{prop}
$\uqlg$ is the associative unital algebra over $\C$ generated by 
\be (x^\pm_{i,n})_{i \in I, n \in \mathbb Z}, \quad\quad (k_i^{\pm 1})_{i \in I}, 
\quad\quad  (\Phi_{i, r})_{i \in I , r \in \Z},\nn\ee
subject to the following relations: 
$k_ik_i^{-1}=1$, 
\be \Phi_{i,0} = \frac{k_i - k_i^{-1}}{q_i-q_i^{-1}},\label{Phidef3}\ee
\be \left[\Phi_i(u),\Phi_j(v)\right] = 0,\quad \left[\Phi_i(u),k_j\right] = 0,\quad k_ik_j =k_j k_i, \label{Phikrel}\ee
\be \left [x^+_i(u), x^-_j(v) \right ] = \delta_{ij} \delta\left(\frac u v\right) \Phi_i\left(\frac 1 u\right), \label{xpxmPhirel}\ee    
\be k_i x^\pm_j(v) = x^\pm_j(v) k_i q^{\pm B_{ij}}, \label{kxpxmrel}\ee
\bea 
(u- q^{\pm B_{ij}}v)\Phi_i(1/u)\,  x^\pm_j(v) &=&  \, (q^{\pm B_{ij}}u- v) \, x^\pm_j(v) \, \Phi_i(1/u) , \label{Phirel}\\    
 \left ( u-q^{\pm B_{ij}}v \right )x^\pm_i(u) \, x^\pm_j(v) &=&     
\left (q^{\pm B_{ij}}u-v \right)\, x^\pm_j(v) \, x^\pm_i(u) \label{xpxmrel}\eea    
together with the Serre relations as above. 
\end{prop}
\begin{proof} 
  The defining relations (\ref{phiphirel}--\ref{Serre}) involve the $h_{i,r}$ only through the $\phi^\pm_{i,\pm r}$. For every $i\in I$, each $h_{i,\pm r}$ can be expressed in terms of $(\phi^\pm_{i,\pm s})_{s\leq r}$ using (\ref{phiu}) and an induction on $r>0$. So we are free to regard the $(\phi^\pm_{i,\pm r})_{i\in I,r\in \Z_{\geq 1}}$ as generators in place of the $(h_{i,r})_{i\in I,r\in \Z_{\neq 0}}$, and then in turn to replace the $(\phi^\pm_{i,\pm r})_{i\in I,r\in \Z_{\geq 1}}$ by $(\Phi_{i,r})_{i\in I, r\in \Z_{\neq 0}}$ according to (\ref{Phidef2}). For later convenience, we also include $(\Phi_{i,0})_{i\in I}$ as generators; consequently we must impose the linear relation (\ref{Phidef3}) as a defining relation.

It is clear that (\ref{Phikrel}) and (\ref{xpxmPhirel}) are equivalent to, respectively, (\ref{phiphirel}) and (\ref{x+x-}). The relation (\ref{xpxmrel}) is copied verbatim. 
It remains to check that (\ref{kxpxmrel}--\ref{Phirel}) is equivalent to (\ref{phiprel}--\ref{phimrel}). 
The $u^0$ terms of (\ref{phiprel}) and (\ref{phimrel}) are, respectively, the relations involving $k_i^{+1}$ and $k_i^{-1}$ in (\ref{kxpxmrel}). As written, (\ref{phiprel}) and (\ref{phimrel}) are equations of formal series in $\uqlg[[u,v,v^{-1}]]$ and $\uqlg[[u^{-1},v,v^{-1}]]$ respectively.\footnote{Consequently (\ref{phiprel}), for example, is equivalent to $\phi^+_i(u)\,  x^\pm_j(v) =  \, \frac{q^{\pm B_{ij}}- uv}{1- q^{\pm B_{ij}}uv} \, x^\pm_j(v) \, \phi^+_i(u)$, with the understanding that one must expand $\frac{q^{\pm B_{ij}}- uv}{1- q^{\pm B_{ij}}uv}$ for small $u$ and equate powers of $u$. 
} But we may regard them both as equations of formal series in $\uqlg[[u,u^{-1},v,v^{-1}]]$, c.f. the definition (\ref{Phidef}-\ref{Phidef2}) of $\Phi_i(u)$ above, 
\bea    
(1- q^{\pm B_{ij}}uv)\phi^+_i(u)\,  x^\pm_j(v) &=&  \, (q^{\pm B_{ij}}- uv) \, x^\pm_j(v) \, \phi^+_i(u) , \nn\\    
(1- q^{\pm B_{ij}}uv)\phi^-_i(u)\,  x^\pm_j(v) &=&  \, (q^{\pm B_{ij}}-uv) \, x^\pm_j(v) \, \phi^-_i(u) \nn\eea
and subtract one from the other to find 
(\ref{Phirel}).  
Finally, we check that all remaining relations in (\ref{phiprel}--\ref{phimrel}) can be recovered from (\ref{kxpxmrel}) and (\ref{Phirel}). The latter unpacks to give
\be \Phi_{i,r} x^\pm_{j,s} - q^{\pm B_{ij}} \Phi_{i,r-1} x^\pm_{j,s-1} = 
  q^{\pm B_{ij}} x^\pm_{j,s} \Phi_{i,r} - x^\pm_{j,s-1} \Phi_{i,r-1}. \nn\ee
For all $r\in \Z_{\geq 2}$ this is 
\be \phi^+_{i,r} x^\pm_{j,s} - q^{\pm B_{ij}} \phi^+_{i,r-1} x^\pm_{j,s-1} = q^{\pm B_{ij}} x^\pm_{j,s} \phi^+_{i,r} - x^\pm_{j,s-1} \phi^+_{i,r-1},\nn\ee
while for $r=1$ it is
\be \phi^+_{i,1} x^\pm_{j,s} - q^{\pm B_{ij}} (k_i-k_i^{-1}) x^\pm_{j,s-1} = q^{\pm B_{ij}} x^\pm_{j,s} \phi^+_{i,1} - x^\pm_{j,s-1} (k_i-k_i^{-1}),\nn\ee
and given that $q^{\pm B_{ij}} k^{-1}_i x^\pm_{j,s-1} = x^\pm_{j,s-1} k^{-1}_i$ by (\ref{kxpxmrel}) and $k_ik_i^{-1}=1$, this is
\be \phi^+_{i,1} x^\pm_{j,s} - q^{\pm B_{ij}} k_i x^\pm_{j,s-1} = q^{\pm B_{ij}} x^\pm_{j,s} \phi^+_{i,1} - x^\pm_{j,s-1} k_i.\nn\ee
Thus we recover all the relations in (\ref{phiprel}). Similarly, by considering $r\in \Z_{\leq 0}$ we recover all the relations in (\ref{phimrel}).
\end{proof}



\subsection{Rational $\ell$-weights and $q$-characters}\label{lweightssec} 
Let $\cc$ denote the category whose objects are finite-dimensional representations of $\uqlg$ and whose morphisms are $\uqlg$-module maps. 

Every representation $V\in\Ob(\cc)$ is a direct sum of its generalized eigenspaces for the mutually commuting $\phi_{i,r}^\pm$:
\be V = \bigoplus_{\bs\gamma} V_{\bs\gamma}\,\,, \qquad \bs\gamma = (\gamma_{i,\pm r}^\pm)_{i\in I, r\in \Z_{\geq 0}}, \quad \gamma_{i,\pm r}^\pm \in \mathbb C\nn\ee 
where \be V_{\bs\gamma} = \{ v \in V : \exists k \in \mathbb N, \,\, \forall i \in I, m\geq 0 \quad \left(
  \phi_{i,\pm m}^\pm - \gamma_{i,\pm m}^\pm\right)^k \on v = 0 \} \,. \nn\ee
If $\dim (V_{\bs\gamma}) >0$, $\bs\gamma$ is called an \emph{$\ell$-weight} of $V$. Given $v\in V$ we write $v_{\bs \gamma}$ for the component of $v$ in $V_{\bs \gamma}$.
Note that $\gamma_{i,0}^+\gamma_{i,0}^-=1$ necessarily. Let us write the $\ell$-weight $\bs \gamma$ as a formal series
\be \gamma_i^\pm(u) := \sum_{r=0}^\8 u^{\pm r} \gamma_{i,\pm r}^\pm .\nn\ee
Whenever $\bs\gamma$ is an $\ell$-weight of a finite-dimensional representation (and more generally whenever $\bs\gamma$ is an $\ell$-weight of a representation in the larger category $\hat\cat$ of \cite{MY12}) the series $\gamma_i^+(u)$ and $\gamma_i^-(u)$ are the Laurent expansions, about $0$ and $\8$ respectively, of a complex-valued rational function $\gamma_i(u)\in \C(u)$ with the property that $\gamma_i(0)\gamma_i(\8)=1$. Following \cite{MY12}, we call such $\ell$-weights \emph{rational}, and henceforth we shall not distinguish between a rational $\ell$-weight and its corresponding tuple of rational functions. 

Rational $\ell$-weights form a abelian group, the group operation being (component-wise) multiplication of (tuples of) rational functions. 
Let $\mc R$ denote this group and $\Z\mc R$ its integral group ring. 
The $q$-character of a $\uqlg$-module $V$ is by definition  the formal sum
\be \chi_q(V) := \sum_{\bs\gamma\in\mc R} \dim(V_{\bs\gamma}) \bs\gamma \in \Z\mc R.\nn\ee
The $\ell$-weights of finite-dimensional representations actually \cite{FR} belong to the subgroup of $\mc R$ generated by rational $\ell$-weights $Y_{i,a}$, $i\in I$, $a\in \Cx$, defined by
\be\left(Y_{i,a}\right)_i(u) := \,q_i \frac{1 - q_i^{-2}ua}{1-ua}, \qquad \left(Y_{i,a}\right)_j(u) := 1 \text{ for all } j\neq i.  \nn\ee 
The $q$-character of a finite-dimensional representation thus belongs to the ring $\Z[Y_{i,a},Y_{i,a}^{-1}]_{i\in I,a \in \Cx}$ of formal Laurent polynomials.
Moreover, when $\uqlg$ is endowed with the standard Hopf algebra structure, the $q$-character map defines an injective homomorphism of rings \cite{FR} \be\label{chiinj}\chi_q:\groth\cc\to \Z[Y_{i,a},Y_{i,a}^{-1}]_{i\in I,a\in \Cx}\ee from the Grothendieck ring $\groth\cc$. 

An $\ell$-weight in $\Z[Y_{i,a},Y_{i,a}^{-1}]_{i\in I,a \in \Cx}$ is called \emph{dominant} if it is in $\Z[Y_{i,a}]_{i\in I,a\in \Cx}$. 
A vector $v$ of a $\uqlg$-module is an \emph{$\ell$-weight vector} of $\ell$-weight $\bs\gamma$ if $\phi_{i,\pm r}^\pm \on v= v\gamma_{i,\pm r}^\pm$, for all $i\in I$, $r\in \Z_{\geq 0}$, and it is a \emph{highest} $\ell$-weight vector if in addition $x_{i,r}^+ \on v=0$ for all $i\in I$, $r\in \Z$. We write $\L(\bs \gamma)$ for the irreducible quotient of $\uqlg\on v$. It is known \cite{CP94} that the map $\bs\gamma\mapsto \L(\bs\gamma)$ is a bijection from the set of dominant $\ell$-weights to the set of isomorphism classes of irreducible finite-dimensional $\uqlg$-modules, i.e. the simple objects of $\cc$.

For each $j\in I$ and $a\in \Cx$, define $A_{j,a} \in \mc R$ by \be (A_{j,a})_{i}(u) = q^{B_{ji}} \frac{1-q^{-B_{ji}} au}{1-q^{B_{ji}}au}\nn\ee for each $i\in I$. 
We call each $A_{j,a}$ a \emph{simple $l$-root}. 
The $A_{j,a}$ are algebraically independent. (The reader should be warned that in \cite{FR,FM} $A_{j,a}$ was instead labelled $A_{j,aq_j}$.)



\section{Motivation, and the categories $\cc_\P$}\label{msec}
It will be important for us that the raising/lowering operators $x_{i,r}^\pm$ of $\uqlg$ act in a rather specific way in all finite-dimensional representations (actually, in all $\ell$-weight modules). Proposition \ref{stepprop} establishes this property. Then, in \S\ref{catsec}, we introduce the subcategories $\cc_\P$ of $\cc$ that we shall work with throughout the rest of the paper.  
\subsection{Action of the raising/lowering operators between $\ell$-weight spaces} 
\begin{prop}\label{stepprop} Let $V\in \Ob(\cc)$.
Pick and fix any $i\in I$. Let $(\bs \mu,\bs \nu)$ be a pair of $\ell$-weights of $V$ such that 
$x^\pm_{i,r}(V_{\bs\mu}) \cap V_{\bs\nu} \neq \{0\}$ for some $r\in \Z$.  
Then:
\begin{enumerate}[(i)]
\item $\bs\nu  =  \bs \mu A_{i,a}^{\pm 1} $ for some $a\in \Cx$, and moreover 
\item there exist bases  $(v_k)_{1\leq k\leq \dim(V_{\bs \mu})}$ of $V_{\bs \mu}$ and $(w_\ell)_{1\leq \ell\leq \dim(V_{\bs \nu})}$ of $V_{\bs \nu}$, and complex polynomials $P^\pm_{k,\ell}(z)$, with $\deg(P^\pm_{k,\ell}) \leq k+\ell-2$, such that
\be(x^\pm_{i}(z) \on v_k)_{\bs \nu}= \sum_{\ell=1}^{\dim(V_{\bs \nu})} w_\ell P^\pm_{k,\ell}\!\left(\pd a\right) \delta\left(\frac{a}{z}\right).\nn\ee
\end{enumerate}
\end{prop}
\begin{proof}
Let $(v_k)_{1\leq k\leq \dim V_{\bs \mu}}$ be a basis of $V_{\bs \mu}$ in which the action of the $\phi_{j, r}^+$ is upper-triangular, in the sense that for all $j\in I$ and $1\leq k\leq \dim V_{\bs \mu}$, 
\be(\phi_j^+(u)- \mu_{j}^+(u))\on v_k = \sum_{k'<k} v_{k'} \xi^{+,k,k'}_j(u),\nn \ee
for certain formal series $\xi^{+,k,k'}_j(u) \in u\C[[u]]$. (The leading order is $u^1$: recall that $\phi_{j,0}^\pm=k_j^{\pm 1}$ act diagonally.)
Let $(w_k)_{1\leq k\leq \dim V_{\bs \nu}}$ be a basis of $V_{\bs \nu}$ in which the action of the $\phi_{j,r}^+$ is lower-triangular, in the sense that for all $j\in I$ and $1\leq \ell\leq \dim V_{\bs \nu}$, 
\be\nn(\phi_j^+(u)-\nu_j^+(u))\on w_{\ell} = \sum_{\ell'>\ell} w_{\ell'} \zeta_{j}^{+,\ell,\ell'}(u),\nn\ee
for certain formal series $\zeta^{+,\ell,\ell'}_j(u) \in u\C[[u]]$.

Consider for definiteness the case of $x^+_{i}(z)$ ($x^-_i(z)$ is similar).
For all $1\leq k\leq \dim(V_{\bs \mu})$, 
\be (x_{i}^+(z) \on v_k)_{\bs \nu} = \sum_{\ell=1}^{\dim(V_{\bs \nu})} \lambda_{k,\ell}(z) w_\ell \nn\ee
for some formal series $\lambda_{k,\ell}(z)\in \C[[z,z^{-1}]]$ for each $\ell$, $1\leq \ell \leq \dim(V_{\bs \nu})$.

By (\ref{phiprel}), we have 
\be (q^{B_{ij}}-uz) x^+_i(z) \left(\phi^+_j(u) - \mu^+_j(u) \right) \on v_k = \left( (1-q^{B_{ij}}uz)\phi^+_j(u) -  (q^{B_{ij}}-uz)  \mu^+_j(u) \right) x^+_i(z)\on v_k\nn\ee
and, on resolving this equation in the basis above of $V_{\bs \nu}$ and taking the $w_{\ell}$ component,
\begin{align}(q^{B_{ij}}-uz)\sum_{k'=1}^{k-1} \xi^{+,k,k'}_j(u) \lambda_{k',\ell}(z) &= \left( (1-q^{B_{ij}}uz)\nu^+_j(u) -  (q^{B_{ij}}-uz) \mu^+_j(u) \right) \lambda_{k,\ell}(z) \nn\\
&\phantom = + (1-q^{B_{ij}}uz) \sum_{\ell'=1}^{\ell-1} \lambda_{k,\ell'}(z) \zeta_j^{+,\ell',\ell}(u) 
\label{XAeqn}.\end{align}
Suppose $(x_{i}^+(z) \on V_{\bs \mu})_{\bs \nu} \neq 0$. Then there is a smallest $K$ such that $(x_{i}^+(z) \on v_K)_{\bs \nu} \neq 0$ and then a smallest $L$ such that $\lambda_{K,L}(z)\neq 0$. So (\ref{XAeqn}) gives, in particular,
\be 0= \left( (1-q^{B_{ij}}uz)\nu^+_j(u) -  (q^{B_{ij}}-uz)\mu^+_j(u) \right) \lambda_{K,L}(z) 
\label{Xeqn}.\ee
This must hold for all $j\in I$. For each $j\in I$, (\ref{Xeqn}) is an equation of the form $0=\lambda_{K,L}(v) \sum_{n=0}^\8 u^n (b^{(i)}_n + c^{(i)}_nv)$  for the formal series $\lambda_{K,L}(v)\in \C[[v,v^{-1}]]$,  with $b^{(i)}_n,c^{(i)}_n\in \C$ for all $n\in\Z_{\geq 0}$. 
There are non-zero solutions if and only if there is an $a\in\Cx$ such that $b^{(i)}_n/c^{(i)}_n=-a$ for all $n\in \Z_{\geq 0}$ and all $j\in I$. That is,
\be \nu^+_j(u) \left(\mu^+_j(u)\right)^{-1} = q^{B_{ij}} \frac{1-q^{-B_{ij}} u a}{1-q^{B_{ij}}u a} =: (A_{i,a})_j(u)\nn\ee
as an equality of power series in $u$. 

This establishes (i), and equation (\ref{XAeqn}) then rearranges to give
\begin{align} \frac{(1-q^{2B_{ij}})\mu^+_j(u)u}{1-q^{B_{ij}}ua} (z-a)\lambda_{k,\ell}(z) &= (q^{B_{ij}}-uz)\sum_{k'=1}^{k-1} \xi^{+,k,k'}_j(u) \lambda_{k',\ell}(z)  \nn\\&\phantom = - (1-q^{B_{ij}}uz) \sum_{\ell'=1}^{\ell-1} \lambda_{k,\ell'}(z) \zeta_j^{+,\ell',\ell}(u)
\label{XBeqn}.\end{align}
Part (ii) now follows from (\ref{XBeqn}) by an induction on $k+\ell$: in  the base case $k+\ell=2$, the right hand side of equation (\ref{XBeqn}) is zero, so the equation for $\lambda_{1,1}$ is $ (z-a) \lambda_{1,1}(z)=0$,
whose solutions are $\lambda_{1,1}(z) = P_{1,1} \delta(a/z)$, with $P_{1,1}\in \C$. 
For the inductive step, pick $(k,\ell)$ and suppose (ii) is true for all $(k',\ell')$ such that $k'+\ell'<k+\ell$. Consider $\lambda_{k,\ell}$.
It is enough to consider the equation obtained by taking the leading power of $u$ in (\ref{XBeqn}), which is $u^1$. This equation is
\be (1-q^{2B_{ij}})\mu_{j,0}^+ (z-a)\lambda_{k,\ell}(z) =q^{B_{ij}}\sum_{k'=1}^{k-1} \lambda_{k',\ell}(z) \xi_{j,1}^{+,k,k'} -\sum_{\ell'=1}^{\ell-1} \lambda_{k,\ell'}(z) \zeta_{j,1}^{+,\ell',\ell}  ,\nn\ee
and every solution is of the form $\lambda_{k,\ell}(z) = P_{k,\ell}( \pd a) \delta(\frac{a}{z})$ with 
\be \nn\deg(P_{k,\ell}) \leq 1+\max_{\substack{k'<k\\\ell'<\ell}} \left(\deg(P_{k',\ell'})\right).\ee (To see this, note that $(z-a)\delta(\frac z a) = 0$,  $(z-a) \pd a \delta(\frac{a}{z}) = \delta(\frac{z}{a})$, and more generally for all $m\in\Z_{\geq 1}$, $(z-a)\frac 1 {m!}( \pd a)^m \delta(\frac z a) = \frac 1{(m-1)!}(\pd a)^{m-1}  \delta(\frac za)$.) This completes the inductive step.
\end{proof}
\begin{rem*} Part (i) in the above was proved in \cite{MY12}. \end{rem*}

\subsection{Categories $\cc_\P$ of finite dimensional representations}\label{catsec}
To get more control over the $\ell$-roots $A_{i,a}$ that can appear in Proposition \ref{stepprop}, it is useful to work not with the category $\cc$ of \emph{all} finite dimensional representations of $\uqlg$, but rather with smaller categories in which only rational $\ell$-weights with poles at certain prescribed points are allowed, as follows.

\begin{defn}\label{cpdef} Pick, for each $i\in I$, a set $\P_i\subset \Cx$ consisting of finitely many (pairwise distinct) points. Given such a tuple $\P= (\P_i)_{i\in I}$, let $\cc_\P$ denote the full subcategory of $\cc$ such that for any  $\ell$-weight $\bs\gamma= (\gamma_i(u))_{i\in I}$ of any representation $V\in \Ob(\cc_\P)$, and for each $i\in I$, the rational function $\gamma_i(u)$ has no poles lying outside the set $\{a^{-1}: a\in \P_i\}$.
\end{defn}

Equivalently, $\cc_\P$ is the full subcategory of $\cc$ whose objects have $q$-characters in $\Z[Y_{i,a},Y_{i,aq_i^{2}}^{-1}]_{i\in I, a\in \P_i}$. 

The following lemma shows that the definition could also be phrased in terms of the allowed $\ell$-roots $A_{i,a}$. (This is not trivial, since $(A^{\pm 1}_{i,a})_j(u)$ has no pole at $a^{-1}$.)
\begin{lem}\label{ainP} Let $V$ be a representation in $\cc_\P$ and $A_{i,a}$ a simple $\ell$-root. If $\bs \mu$ and $\bs\mu A_{i,a}^{-1}$ are both $\ell$-weights of $V$ then $a\in \P_i$. 
\end{lem}
\begin{proof}
This follows from Theorem 5.1 of \cite{FM}, which states that $\chi_q(\cc)$ is contained in (in fact, is equal to) the intersection $\bigcap_{i\in I} K_i$ of the sets $K_i := \Z[Y_{j,a}^{\pm 1}]_{j\neq i; a\in \Cx}\otimes \Z[Y_{i,b} + Y_{i,b} A_{i,b}^{-1}]_{b\in \Cx}$.
\end{proof}

By (\ref{chiinj}), $\cc_{\P}$ is closed under taking submodules, quotients, finite direct sums (and tensor products, when $\cc$ is made into a tensor category using the standard Hopf algebra structure on $\uqlg$). 

\begin{rem} The categories $\mathscr{C}_\ell$ of the paper \cite{HernandezLeclerc} are of this form. (In \cite{HernandezLeclerc}, the definition of $\mathscr{C}_{\ell}$ is given in terms of the allowed highest $\ell$-weights.)\end{rem}


\begin{rem}
Roughly speaking, the choices of $\P$ for which the categories $\cc_\P$ are non-trivial are those in which the $\P_i$ contain sufficiently long subsequences of geometric progressions in $q$. 
In much of \S\ref{appssec} we specialize to choices of the form $\P_i = \{aq^k:k\in \Z, 0\leq k \leq K\}$ for all $i\in I$, with $a\in \Cx$ and $K\in \Z_{\geq 0}$ fixed. 
The reader may find it helpful to keep such a choice of $\P$ in mind throughout.
\end{rem}

\begin{rem}
\label{blockrem}
The restriction to finite sets of allowed poles is for technical simplicity in what follows. Suitably modified, the main arguments should go through for countable sets without accumulation points in $\Cx$. Of particular interest are the blocks of the category $\cc$. In simply-laced cases these are of the form $\cc_\P$ with $\P_i = aq^{2\Z+ c(i)}$ for some $a\in \Cx$, where $c:I\to \{0,1\}$ is a two-colouring of the Dynkin diagram \cite{CMblocks}. In the non-simply-laced cases the blocks are still of the form $\cc_\P$ for certain countable sets $\P_i$, $i\in I$.
\end{rem}


The following is essentially a corollary of Proposition \ref{stepprop}.

\begin{prop}\label{idea} Let $V\in \Ob(\cc_{\P})$, with $\rho:\uqlg\to \End(V)$ the representation homomorphism.  Let $M=2\max_{\bs\mu}\dim(V_{\bs\mu})$. Then there exist linear maps $E^\pm_{i,a,m}$ and $H_{i,a,m}$ in $\End(V)$, for each $i\in I$, $a\in \P_i$ and $0\leq m\leq M$, such that
\begin{align} \rho(x^\pm_i(z))  &= \sum_{a\in \P_i}\sum_{m=0}^M E^\pm_{i,a,m} \frac{a^m}{m!} \left(\pd a\right)^m \delta\left(\frac a z\right)  , \label{Emap}\\
 \rho(\Phi_{i}(1/z))  &= \sum_{a\in \P_i}\sum_{m=0}^M H_{i,a,m}\frac{a^m}{m!} \left(\pd a\right)^m \delta\left(\frac a z\right)  .\label{Hmap}\end{align}
Moreover, $E_{i,a,m}^\pm(V_{\bs\mu}) \subseteq V_{\bs\mu A_{i,a}^{\pm 1}}$ (and 
$H_{i,a,m}(V_{\bs\mu}) \subseteq V_{\bs\mu}$) for each $\ell$-weight $\bs\mu$ of $V$. 
\end{prop}
\begin{proof}

Pick an $i\in I$.
 
We define the required maps $E^\pm_{i,a,m}$ by giving their restrictions to $\Hom(V_{\bs\mu,},V_{\bs\nu})$ for each pair of $\ell$-weights $\bs\mu$ and $\bs\nu$ of $V$. If the restriction of $\rho(x^\pm_{i}(z))$ to $\Hom(V_{\bs\mu,},V_{\bs\nu})[[z,z^{-1}]]$ is zero then we define the restriction of $E^\pm_{i,a,m}$ to $\Hom(V_{\bs\mu,},V_{\bs\nu})$ to be zero too, for each $a\in \P_i$. Otherwise, we use Proposition \ref{stepprop}. Part (i) asserts that there is an $a\in \Cx$ such that $\bs\nu=\bs\mu A_{i,a}^{\pm 1}$. By Lemma \ref{ainP}, $a\in \P_i$.
Suppose, in part (ii) of Proposition \ref{stepprop}, that $P^\pm_{k,\ell}(x) = \sum_{m=0}^{k+\ell-2} P^\pm_{k,\ell,m} \frac{a^m}{m!}x^m$ and set $P^\pm_{k,\ell,m}=0$ for all $m> k+\ell-2$. Then the restriction of $E^\pm_{i,a,m}$ to $\Hom(V_{\bs\mu},V_{\bs\nu})$ is given by  
\be (E_{i,a,m}^\pm\on v_k)_{\bs\nu} = \sum_{\ell=1}^{\dim(V_{\bs\nu})} P^\pm_{k,\ell,m} w_\ell\nn,\ee
while for all $b\in \P_i\setminus\{a\}$ the restriction of $E^\pm_{i,b,m}$ to $\Hom(V_{\bs\mu},V_{\bs\nu})$ is zero. 

Note that $E_{i,a,m}^\pm(V_{\bs\mu}) \subseteq V_{\bs\mu A_{i,a}^{\pm 1}}$ by construction.

Next we show that the maps $H_{i,a,m}$ exist too. 
From (\ref{Emap}) we have in particular that $\rho(x_{i,0}^-) = \sum_{a\in\P_i} E_{i,a,0}^-$. It follows from relation (\ref{xpxmPhirel}) that $[x_i^+(u),x_{i,0}^-] = \Phi_i(1/u)$. Thus, given (\ref{Emap}), we have
\be \rho(\Phi_i(1/u)) = \sum_{a,b\in \P_i} \sum_{m=0}^M  \left(\frac{a^m}{m!} \left(\pd a\right)^m \delta\left(\frac a u\right) \right) [E^+_{i,a,m},E^-_{i,b,0}].\nn\ee
Note that $[E^+_{i,a,m},E^-_{i,b,n}](V_{\bs\mu}) \subset V_{\bs\mu A_{i,a} A_{i,b}^{-1}}$ for all $\ell$-weights $\bs\mu$ of $V$, whereas by definition of $\ell$-weight, $\rho(\Phi_i(1/u))(V_{\bs\mu})\subseteq V_{\bs\mu}$.  So the terms with $a\neq b$ on the right must sum to zero and can be dropped. Hence indeed
\be \rho(\Phi_i(1/u)) = \sum_{a,b\in \P_i} \sum_{m=0}^M  \left(\frac{a^m}{m!} \left(\pd a\right)^m \delta\left(\frac a u\right) \right) H_{i,a,m}\nn\ee
as required, where we define $H_{i,a,m} := [E^+_{i,a,m},E^-_{i,a,0}]\in \End(V)$. \end{proof}

\section{The algebra $\alg$ and statement of main result}\label{mainsec}
We first define the algebra $\alg$ that will appear in the statement of Theorem \ref{mainthm} below. The defining relations are given in \S\ref{algdef}, and we then need to take a certain completion, as discussed in \S\ref{completion}.

\subsection{Defining relations of $\alg$}
\label{algdef} 
Let $\P=(\P_i)_{i\in I}$ be a tuple as in Definition \ref{cpdef}. We define an algebra $\alg$, depending on this choice of $\P$, and also on $\g$ and $q$, as follows. 
Let $\alg$ be the associative unital algebra over $\C$ generated by \be \K_i^{\pm 1},\quad \E^\pm_{i,a,m}, \quad \H_{i,a,m}, \quad i\in I,\, a\in \P_i,\, m\in \Z_{\geq 0},\nn \ee
subject to the following relations. 
For all $i,j\in I$, $a\in \P_i$, $b\in \P_j$ and $m,n\in \Z_{\geq 0}$,
\be \K_i\K_i^{-1} = 1, \quad \K_i \E^\pm_{j,b,m} = \E^\pm_{j,b,m} q^{\pm B_{ij}} \K_i, \quad \K_i \K_j = \K_j \K_i\ee\be  \K_i \H_{j,b,m} = \H_{j,b,m} \K_i,\quad \H_{i,a,m} \H_{j,b,n} = \H_{j,b,n} \H_{i,a,m} \label{Krels}\ee
\be \left[ \E^+_{i,a,m}, \E^-_{j,b,n} \right] = \delta_{ij} \delta_{a,b} \H_{i,a,m+n}\label{EpEmrel} \ee
\bea
&&(a-b q^{\pm B_{ij}}) \E^\pm_{i,a,m} \E^\pm_{j,b,n} + a\E^\pm_{i,a,m+1} \E^\pm_{j,b,n} - bq^{\pm B_{ij}} \E^\pm_{i,a,m} \E^\pm_{j,b,n+1} \nn\\
&=& (aq^{\pm B_{ij}}-b ) \E^\pm_{j,b,n} \E^\pm_{i,a,m}  + aq^{\pm B_{ij}} \E^\pm_{j,b,n} \E^\pm_{i,a,m+1}  -  b\E^\pm_{j,b,n+1} \E^\pm_{i,a,m} \label{EErel}
\eea
\bea
&&(a-b q^{\pm B_{ij}}) \H_{i,a,m} \E^\pm_{j,b,n} + a\H_{i,a,m+1} \E^\pm_{j,b,n} - bq^{\pm B_{ij}} \H_{i,a,m} \E^\pm_{j,b,n+1} \nn\\
&=& (aq^{\pm B_{ij}}-b ) \E^\pm_{j,b,n} \H_{i,a,m}  + aq^{\pm B_{ij}} \E^\pm_{j,b,n} \H_{i,a,m+1}  -  b\E^\pm_{j,b,n+1} \H_{i,a,m} \label{EHrel}
\eea
\be \sum_{a\in \P_i} \H_{i,a,0} = \frac{\K_i-\K_i^{-1}}{q_i-q_i^{-1}}\label{HKrel}.\ee
In addition, for each $i\neq j$, 
set $s=1-C_{ij}$; then for all $m_1,\dots,m_s,n\in \Z_{\geq 0}$ and all $a\in \P_j$ such that $\{aq^{\mp B_{ij} \mp (t-1) B_{ii}}: 1\leq t\leq s\}\subset \P_i$, we impose the relation
\be\begin{split} \sum_{r=0}^{s}(-1)^r
{\binom s r}_{\!q_i} 
\E^\pm_{i,aq^{\mp B_{ij} \mp r B_{ii}},m_{r+1}}\ldots
\E^\pm_{i,aq^{\mp B_{ij} \mp (s-1) B_{ii}},m_s}  \E^\pm_{j,a,n} \E^\pm_{i,aq^{\mp B_{ij}},m_1} \ldots \E^\pm_{i,aq^{\mp B_{ij}  \mp (r-1) B_{ii} },m_{r}} =0.\label{cycleser}\end{split}\ee

\begin{rem}\label{serrerem} The relations (\ref{cycleser}) are analogs of the Serre relations. In contrast to  (\ref{Serre}), they do not involve a sum over the symmetric group $\Sigma_s$. It is interesting to note that they take the same form as the Serre relations for the finite-type quantum group, but ``decorated'' with shifts in ``special position'', cf. \S\ref{sticksec} below.
For clarity, let us list their explicit form for $\E^-$, case-by-case: 
\bea
\begin{tikzpicture}[baseline =-5,scale=.8] 
\filldraw[fill=white] (1,0) circle (1mm) node [below] {$i$};    
\filldraw[fill=white] (2,0) circle (1mm) node [below] {$j$};    
\end{tikzpicture}
&:& 0= \E^-_{j,b,n} \E^-_{i,b,m} - \E^-_{i,b,m} \E^-_{j,b,n}    \nn\\\nn\\
\begin{tikzpicture}[baseline =-5,scale=.8] 
\draw[thick] (1,0) -- (2,0);
\filldraw[fill=white] (1,0) circle (1mm) node [below] {$i$};    
\filldraw[fill=white] (2,0) circle (1mm)  node [below] {$j$};    
\end{tikzpicture}
&:& 0= \E^-_{j,b,n} \E^-_{i,bq^{-1},m_1} \E^-_{i,bq,m_2} - [2]_q
\E^-_{i,bq,m_2} \E^-_{j,b,n} \E^-_{i,bq^{-1},m_1}+ 
 \E^-_{i,bq^{-1},m_1} \E^-_{i,bq,m_2} \E^-_{j,b,n}     \nn\\\nn\\
\begin{tikzpicture}[baseline =-5,scale=.8] 
\draw[double,thick] (1,0) -- (2,0);
\draw[thick] (1.6,.2) -- (1.4,0) -- (1.6,-.2);
\filldraw[fill=white] (1,0) circle (1mm) node [below] {$i$};    
\filldraw[fill=white] (2,0) circle (1mm)  node [below] {$j$};    
\end{tikzpicture}
&:& 0= \E^-_{j,b,n} \E^-_{i,bq^{-2},m_1} \E^-_{i,b,m_2} \E^-_{i,bq^2,m_3} 
    -[3]_q \E^-_{i,bq^2,m_3}  \E^-_{j,b,n} \E^-_{i,bq^{-2},m_1} \E^-_{i,b,m_2} \nn\\
&& \qquad    +[3]_q \E^-_{i,b,m_2} \E^-_{i,bq^2,m_3}  \E^-_{j,b,n} \E^-_{i,bq^{-2},m_1}
    -  \E^-_{i,bq^{-2},m_1} \E^-_{i,b,m_2} \E^-_{i,bq^2,m_3}  \E^-_{j,b,n} 
\nn\\\nn\\
&& 0= \E^-_{i,b,n} \E^-_{j,bq^{-2},m_1} \E^-_{j,bq^2,m_2} - [2]_{q^2}
\E^-_{j,bq^2,m_2} \E^-_{i,b,n} \E^-_{j,bq^{-2},m_1}+ 
 \E^-_{i,bq^{-2},m_1} \E^-_{i,bq^2,m_2} \E^-_{j,b,n} \nn   \eea\bea
\begin{tikzpicture}[baseline =-5,scale=.8] 
\draw[thick] (1,0.075) -- (2,0.075);\draw[thick] (1,-0.075) -- (2,-0.075);\draw[thick] (1,0) -- (2,0);
\draw[thick] (1.6,.2) -- (1.4,0) -- (1.6,-.2);
\filldraw[fill=white] (1,0) circle (1mm) node [below] {$i$};    
\filldraw[fill=white] (2,0) circle (1mm)  node [below] {$j$};    
\end{tikzpicture}
&:& 0= \E^-_{j,b,n} \E^-_{i,bq^{-3},m_1} \E^-_{i,bq^{-1},m_2} \E^-_{i,bq^1,m_3} \E^-_{i,bq^3,m_4}
   - [4]_q \E^-_{i,bq^3,m_4} \E^-_{j,b,n} \E^-_{i,bq^{-3},m_1} \E^-_{i,bq^{-1},m_2} \E^-_{i,bq^1,m_3} \nn\\ &&\qquad
 \!\!\!\!\!\!+ {\binom 4 2}_q \E^-_{i,bq^1,m_3} \E^-_{i,bq^3,m_4} \E^-_{j,b,n} \E^-_{i,bq^{-3},m_1} \E^-_{i,bq^{-1},m_2}  
  - [4]_q  \E^-_{i,bq^{-1},m_2} \E^-_{i,bq^1,m_3} \E^-_{i,bq^3,m_4} \E^-_{j,b,n} \E^-_{i,bq^{-3},m_1}
  \nn\\ &&\qquad\quad +  \E^-_{i,bq^{-3},m_1} \E^-_{i,bq^{-1},m_2} \E^-_{i,bq^1,m_3} \E^-_{i,bq^3,m_4} \E^-_{j,b,n} 
\nn\\\nn\\
&&
0= \E^-_{i,b,n} \E^-_{j,bq^{-3},m_1} \E^-_{j,bq^3,m_2} - [2]_{q^3}
\E^-_{j,bq^3,m_2} \E^-_{i,b,n} \E^-_{j,bq^{-3},m_1}+ 
 \E^-_{i,bq^{-3},m_1} \E^-_{i,bq^3,m_2} \E^-_{j,b,n}. \nn   
\eea
\end{rem}

\subsection{Completion of $\alg$}\label{completion} 
Let $\mc I_0 := \alg$. For each $m\in \Z_{>0}$, let $\mc I_m\subset \alg$ be the two-sided ideal generated by $\E^+_{i,a,n}$, $\E^-_{i,a,n}$ and $\H_{i,a,n}$, $i\in I, a\in \P_i$ and $n\geq m$. These ideals form a strictly decreasing sequence
\be \alg = \mc I_0  \supset \mc I_1  \supset \mc I_2 \supset \dots .\nn\ee

\begin{exmp}\label{spes1} Suppose $a\in \P_i$, $b\in \P_j$.  By relation (\ref{EErel}), 
\bea \E^-_{i,a,0} \E^-_{j,aq^{-B_{ij}},0}  &=& \frac{1}{1-q^{- 2B_{ij}}} \Bigl(- \E^-_{i,a,1} \E^-_{j,aq^{-B_{ij}},0} + q^{-2B_{ij}} \E^-_{i,a,0} \E^-_{j,aq^{-B_{ij}},1} \nn\\
&&\quad\qquad\qquad + q^{- B_{ij}} \E^-_{j,aq^{-B_{ij}},0} \E^-_{i,a,1}  - q^{-B_{ij}} \E^-_{j,aq^{-B_{ij}},1} \E^-_{i,a,0}\Bigr). \nn\eea 
Thus $\E^-_{i,a,0} \E^-_{j,b,0}\in \mc I_1$ if $aq^{-B_{ij}} = b$.
Note that $\E^-_{i,a,0} \E^-_{j,aq^{-B_{ij}},0}\notin \mc I_2$, because the third and final terms of the right-hand side cannot be further re-written in this way (although the first two terms can).
We also have (when $i=j$) that $\E^-_{i,a,0} \E^-_{i,a,0}\in\mc I_1$, because relation (\ref{EErel}) yields
\be \E^-_{i,a,0} \E^-_{i,a,0} = \frac{1}{1-q^{-B_{ii}}} \left( q^{-B_{ii}} \E^-_{i,a,0} \E^-_{i,a,1} - \E^-_{i,a,1} \E^-_{i,a,0} \right) .\ee
\end{exmp}
More generally, inspecting  (\ref{EErel}) and (\ref{EHrel}) one sees that for any finite linear combination $S$ of monomials in the generators of $\alg$, if $S\neq 0$ then there is an $n\in \Z_{\geq 1}$ such that $S\notin \mc I_n$. That is, we have 
\begin{lem} $\bigcap_{n=1}^\8 \mc I_n = \{0\}$. 
\qed \end{lem}

By virtue of this we can, in what follows, work with infinite linear combinations of the form
\be S=\sum_{m=0}^\8 u_m, \qquad u_m\in \mc I_m.\nn\ee
Such a sum is to be interpreted as the sequence of partial sums 
\be S=\left(S_0,S_1,S_2,\dots\right),\qquad S_M = \sum_{m=0}^M u_m,\nn \ee
and we declare that $S=S'$ if and only if $S_M \equiv S'_M \mod \mc I_{M+1}$ for all $M\in \Z_{\geq 0}$.

More formally, the decreasing sequence of ideals $(\mc I_m)_{m\geq 0}$ induces a topology on $\alg$: namely the topology in which for each $x\in \alg$, the sets $\{ x+ \mc I_m :m\in \Z_{\geq 0}\}$ form a base of the open sets containing $x$. This topology is Hausdorff by the above lemma. And henceforth we work in the completion $\lim\limits_{\longleftarrow} \alg/\mc I_n$ of $\alg$ with respect to this topology. By a slight abuse, we continue to write $\alg$ for this completion.

\subsection{Main result}
We now state the main result of the paper. 
\begin{thm}\label{mainthm}
There is a homomorphism of algebras $\theta_{\P}:\uqlg\to \alg$ defined by
$k_i\mapsto \K_i$ and
\begin{align} x^\pm_{i}(z) &\mapsto \sum_{a\in \P_i} \sum_{m\in \Z_{\geq 0}} \E^\pm_{i,a,m} \frac{a^m}{m!}  \left(\pd a\right)^m \delta\left(\frac{a}{z}\right) ,\nn\\
\Phi_i(1/z) &\mapsto \sum_{a\in \P_i}\sum_{m\in \Z_{\geq 0}} \H_{i,a,m}\frac{a^m}{m!} \left(\pd a\right)^m \delta\left(\frac{a}{z}\right)\nn .\end{align}
Moreover, every $V\in \Ob(\cc_\P)$ is the pull-back by $\theta_\P$ of a finite-dimensional representation of $\alg$. 
\end{thm}

\section{Proof of Theorem \ref{mainthm}}\label{proofsec}
This section is devoted to the proof of Theorem \ref{mainthm}. We first re-express the defining relations of $\alg$ in terms of formal generating series in \S\ref{gser}. Then, much of the section is devoted to showing that the Serre relations (\ref{cycleser}) are equivalent to the, a priori stronger, relations (\ref{autser2}) below. This is done in Propositions \ref{mostlyautoprop} and  \ref{cyclicserresufficeprop}, and relies on the use of an identity due to Jing \cite{Jing}. With these preparations complete, Theorem \ref{mainthm} is proved in \S\ref{mainproof}: the existence of the homomorphism in Proposition \ref{thetadef}, and the ``moreover'' part in Proposition \ref{pullbackprop}.   

\subsection{Defining relations expressed through formal series}\label{gser}
For each $\G\in \{\E^+,\E^-,\H\}$, $i\in I$ and $a\in \P_i$, let us introduce a generating series 
\be \G_{i,a}(z) := \sum_{m=0}^\8 \frac{z^m}{m!} \G_{i,a,m} \in \alg[[z]].\label{Gseriesdef}\ee
Given any formal series $c\left(\pd z\right) = \sum_{k=0}^\8 c_k \left(\pd z\right)^k\in \C[[\pd z]]$, the product $c\left(\pd z\right) \G_{i,a}(z)$ is well-defined because for each $m\in \Z$ the coefficient of $z^m$ is a well-defined infinite sum, cf \S\ref{completion}:
\be c\left(\pd z\right) \G_{i,a}(z) = 
 \sum_{m=0}^\8 \frac {z^m}{m!} \sum_{n=0}^\8 c_n \G_{i,a,m+n}.\nn\ee  
More generally, for any formal series $c(\pd{z_1},\dots,\pd{z_K}) = \sum_{p,r\in \Z} c_{p_1,\dots,p_K}\prod_{k=1}^K \left(\pd{z_k}\right)^{p_k} \in \C[[\pd{z_1},\dots,\pd{z_K}]]$, products of the form $c(\pd{z_1},\dots,\pd{z_K}) \G^{(1)}_{i_1,a_1}(z_1)\dots \G^{(K)}_{i_K,a_K}(z_K)$ with each $\G^{(k)}\in \{\E^+,\E^-,\H\}$ are well-defined.

The defining relations (\ref{EpEmrel}--\ref{EHrel}) are equivalent to\footnote{It is perhaps interesting to note that if instead of (\ref{Gseriesdef}) one chose to work with geometric  generating series $\G_{i,a}(z) := \sum_{m=0}^\8 z^m  \G_{i,a,m}$ then the first of these relations would read 
\be \left[\E^+_{i,a}(z), \E^-_{j,b}(w) \right] = 
 \delta_{ij} \delta_{a,b} \frac{z H_{i,a}(z) - wH_{i,a}(w)}{z-w}.\nn\ee
}
\be \left[ \E^+_{i,a}(z), \E^-_{j,b}(w) \right] = \delta_{ij} \delta_{a,b} \H_{i,a}(z+w).\nn\ee
\bea \left( a \left(1 + \pd z\right) - bq^{\pm B_{ij}} \left( 1 + \pd w\right) \right) \E^\pm_{i,a}(z) \E^\pm_{j,b}(w) &=&
\left( a q^{\pm B_{ij}} \left(1 + \pd z\right) - b \left( 1 + \pd w\right) \right) \E^\pm_{j,b}(w) \E^\pm_{i,a}(z).\nn\\
\left( a \left(1 + \pd z\right) - bq^{\pm B_{ij}} \left( 1 + \pd w\right) \right) \H_{i,a}(z) \E^\pm_{j,b}(w) &=&
\left( a q^{\pm B_{ij}} \left(1 + \pd z\right) - b \left( 1 + \pd w\right) \right) \E^\pm_{j,b}(w) \H_{i,a}(z).\nn\eea
Moreover we have the following.
\begin{prop}\label{reEprop}
Whenever $a\neq bq^{\pm B_{ij}}$,
\be  \E^\pm_{i,a}(z) \E^\pm_{j,b}(w)
=
\frac{ a q^{\pm B_{ij}} \left(1 + \pd z\right) - b \left( 1 + \pd w\right) }{ a \left(1 + \pd z\right) - bq^{\pm B_{ij}} \left( 1 + \pd w\right) } 
\E^\pm_{j,b}(w) \E^\pm_{i,a}(z) \nn
\ee
\be  \H_{i,a}(z) \E^\pm_{j,b}(w)
= \frac{ a q^{\pm B_{ij}} \left(1 + \pd z\right) - b \left( 1 + \pd w\right) }{ a \left(1 + \pd z\right) - bq^{\pm B_{ij}} \left( 1 + \pd w\right) } \E^\pm_{j,b}(w) \H_{i,a}(z) 
\nn\ee
(where $\frac{ a q^{\pm B_{ij}} \left(1 + \pd z\right) - b \left( 1 + \pd w\right) }{ a \left(1 + \pd z\right) - bq^{\pm B_{ij}} \left( 1 + \pd w\right) }$ is to be interpreted as a formal series in $\pd z$ and $\pd w$ by regarding them as small and expanding).
\end{prop}
\begin{proof}
Let $C(p,r) = (-1)^{r+1}\binom p r\frac{(bq^{\pm B_{ij}})^r a^{(p-r)}  }{(bq^{\pm B_{ij}}-a)^{p}}$.
By induction, we find that for all $P\in \Z_{\geq 0}$,
\bea \E^\pm_{i,a,m} \E^\pm_{j,b,n}  &=&  \E^\pm_{j,b,n}\E^\pm_{i,a,m}\frac{b-a q^{\pm B_{ij}}}{bq^{\pm B_{ij}}-a }\nn\\
&&{} + \sum_{p=1}^P \sum_{r=0}^p \E^\pm_{j,b,n+r} \E^\pm_{i,a,m+p-r}   \frac{q^{B_{ij}}-q^{-B_{ij}}}{bq^{\pm B_{ij}}-a} \frac{ (p-r)bq^{\pm B_{ij}}+ra}{p}  C(p,r) \nn\\ 
  &&{}    +\sum_{r=0}^{P+1} \E^\pm_{j,b,n+r} \E^\pm_{i,a,m+P+1-r}  \frac{(P+1-r)q^{B_{ij}} + rq^{-B_{ij}}}{P+1}  C(P+1,r) \nn\\ 
&&{} - \sum_{r=0}^{P+1} \E^\pm_{i,a,m+P+1-r} \E^\pm_{j,b,n+r}   C(P+1,r) .\nn\eea
The final two lines belong to $\mc I_{P+1}$. Therefore
\bea \E^\pm_{i,a,m} \E^\pm_{j,b,n} &=& \E^\pm_{j,b,n}\E^\pm_{i,a,m}\frac{b-a q^{\pm B_{ij}}}{bq^{\pm B_{ij}}-a }\nn\\
&&{} + \sum_{p=1}^\8 \sum_{r=0}^p \E^\pm_{j,b,n+r} \E^\pm_{i,a,m+p-r} \frac{q^{B_{ij}}-q^{-B_{ij}}}{bq^{\pm B_{ij}}-a} \frac{ (p-r)bq^{\pm B_{ij}}+ra}{p}  C(p,r) \nn\eea
and hence the result.
\end{proof}
\Roff
\subsection{Sticking graphs and the Serre relations}\label{sticksec}
The ``whenever'' condition in Proposition \ref{reEprop} motivates the following definition.
We say $(i,a)$ \emph{sticks to the left of} $(j,b)$ if $a=bq^{-B_{ij}}$. 
Given a multiset $V$ of elements of $I\times \Cx$, define the \emph{sticking graph} of $V$ to be the directed graph $(V,E)$ with vertex set $V$ and directed edges 
\be  E=\{ (i,a) \to (j,b): (i,a) \text{ sticks to the left of } (j,b) \}.\ee 

Now, and for the remainder of this subsection, pick and fix $i,j\in I$ such that $i\neq j$.
We shall treat the Serre relations involving the nodes $i$ and $j$ of the Dynkin diagram.
Let $\Sigma_s$ be the symmetric group on $s:=1-C_{ij}$ letters, and let $a_1,a_2,\dots,a_s\in \Cx$ and $b\in \Cx$. 

\begin{lem}\label{cyclem}
If there is a $\sigma\in \Sigma_s$ such that for each $t\in \{1,\dots,s\}$,
\be a_{\sigma(t)} = bq^{B_{ij}+ (t-1) B_{ii}},\nn\ee
then the sticking graph of $\left\{ (i,a_1), (i,a_2), \dots, (i,a_s), (j,b)\right\}$ is the directed cycle graph
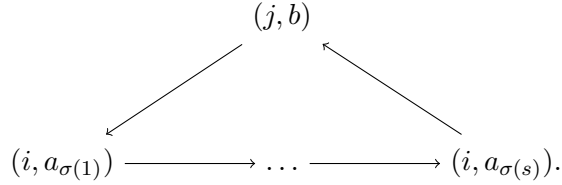
\be\begin{tikzpicture} \matrix (m) [matrix of math nodes, row sep=3em,    
column sep=4em, text height=2ex, text depth=1ex]    
{  &  (j,b) &    \\    
 (i,a_{\sigma(1)})  & \dots & (i,a_{\sigma(s)}). \\  };    
\path[->,font=\scriptsize]    
(m-1-2) edge (m-2-1)    
(m-2-1) edge (m-2-2)    
(m-2-2) edge (m-2-3)
(m-2-3) edge (m-1-2);
\end{tikzpicture}\nn\ee
Otherwise, it has no directed cycles. 
\end{lem}
\begin{proof}
For the first part observe that if the given condition holds then $a_{\sigma(s)}=q^{-B_{ij}}$; to see this note that $B_{ii}C_{ij}=2B_{ij}$. So the sticking graph is indeed the cycle shown. The ``otherwise'' part can be seen by inspecting the Cartan matrices of $A_1\times A_1$, $A_2$, $B_2$, $C_2$ and $G_2$.
\end{proof}

\begin{prop}\label{mostlyautoprop}
Suppose the sticking graph of $\left\{ (i,a_1), (i,a_2), \dots, (i,a_s), (j,b)\right\}$ is not a directed cycle graph. Then the relation (\ref{EErel}) implies the relation 
\be \sum_{\pi\in\Sigma_s}\sum_{r=0}^s(-1)^r
{\binom s r}_{q_i} 
\E^\pm_{i,a_{\pi(1)}}(z_{\pi(1)})\ldots
  \E^\pm_{i,a_{\pi(r)}}(z_{\pi(r)})  \E^\pm_{j,b}(w) \E^\pm_{i,a_{\pi(r+1)}}(z_{\pi(r+1)})\ldots \E^\pm_{i,a_{\pi(s)}}(z_{\pi(s)}) =0.\label{autser}\ee
\end{prop}
\begin{proof}
Let us consider, for definiteness, $\E^-$; the argument for $\E^+$ is similar. By the preceding lemma, the sticking graph of $\left\{ (i,a_1), (i,a_2), \dots, (i,a_s), (j,b)\right\}$ has no directed cycles. Consequently, there must exist a permutation $\tau\in \Sigma_s$ and a $t\in \{1,\dots,s\}$ such that in the  tuple
\be \left( (i,a_{\tau(1)}), \dots, (i,a_{\tau(t)}), (j,b), (i,a_{\tau(t+1)}), \dots, (i, k_{\tau(s)})\right)\nn\ee
no element sticks to the left of any element preceding it; without loss of generality, suppose that $\tau=\id$. 
It follows that we may use Proposition \ref{reEprop} to express every monomial appearing on the left of (\ref{autser}) in terms of the monomial  
\be\nn \scr M :=\E^-_{i,a_{1}}(z_{1}) \dots \E^-_{i,a_{t}}(z_{t}) \E^-_{j,b}(w) \E^-_{i,a_{t+1}}(z_{t+1}) \dots \E^-_{i,a_{s}}(z_{s}).\ee 
Indeed, the left-hand side of (\ref{autser}) is
\bea &=& \sum_{\pi\in\Sigma_s}\sum_{r=0}^s (-1)^r {\binom s r}_{q_i} 
\prod_{\substack{n<m\\ \pi^{-1}(n) > \pi^{-1}(m)}} 
\frac{ a_m q^{ - B_{ii}} \left(1 + \pd{z_m}\right) - a_n  \left( 1 + \pd{z_n}\right) }{ a_m  \left(1 + \pd{z_m}\right) - a_n q^{- B_{ii}} \left( 1 + \pd{z_n}\right) }\nn\\
&&\times\prod_{\substack{n>t\\ \pi^{-1}(n)\leq r}}
\frac{ a_n q^{- B_{ij}} \left(1 + \pd{z_n}\right) - b \left( 1 + \pd w\right) }{ a_n \left(1 + \pd {z_n}\right) - bq^{- B_{ij}} \left( 1 + \pd w\right) } 
\prod_{\substack{n\leq t\\ \pi^{-1}(n)> r}}
\frac{ bq^{- B_{ij}} \left(1 + \pd w\right) - a_n  \left( 1 + \pd{z_n}\right) }{ b \left(1 + \pd w\right) - a_n q^{- B_{ij}} \left( 1 + \pd{z_n}\right) }\,\, \scr M
\nn\\
&=&
\frac 1 \Delta \sum_{\pi\in\Sigma_s}\sum_{r=0}^s (-1)^r {\binom s r}_{q_i}  \nn\\
&&\times\!\!\!\!\!\!
\prod_{\substack{n<m\\ \pi^{-1}(n) < \pi^{-1}(m)}} 
\left(a_m  \left(1 + \tpd{z_m}\right) - a_n q^{- B_{ii}} \left( 1 + \tpd{z_n}\right) \right)
\!\!\!\!\prod_{\substack{n<m\\ \pi^{-1}(n) > \pi^{-1}(m)}} 
\left(a_m q^{- B_{ii}} \left(1 + \tpd{z_m}\right) - a_n  \left( 1 + \tpd{z_n}\right) \right) \nn\\
&&\times 
\prod_{\substack{n>t\\ \pi^{-1}(n)> r}}
\left( a_n  \left(1 + \tpd {z_n}\right) - bq^{- B_{ij}} \left( 1 + \tpd w\right) \right) 
\prod_{\substack{n>t\\ \pi^{-1}(n)\leq r}}
\left( a_n q^{- B_{ij}} \left(1 + \tpd{z_n}\right) - b \left( 1 + \tpd w\right)\right) \nn\\
&&\times 
\prod_{\substack{n\leq t\\ \pi^{-1}(n)\leq r}}
\left( b \left(1 + \tpd w\right) - a_n q^{- B_{ij}} \left( 1 + \tpd{z_n}\right) \right)
\prod_{\substack{n\leq t\\ \pi^{-1}(n)> r}}
\left( bq^{- B_{ij}} \left(1 + \tpd w\right) - a_n  \left( 1 + \tpd{z_n}\right)\right)\,\,\scr M,
\label{autser1}\eea
where
\bea\Delta &:=& \prod_{n<m} \left(a_m \left(1 + \tpd{z_m}\right) - a_n q^{- B_{ii}} \left( 1 + \tpd{z_n}\right)\right)\nn\\
 &&  \times \prod_{n>t} \left(a_n \left(1 + \tpd {z_n}\right) - bq^{- B_{ij}} \left( 1 + \tpd w\right) \right)
     \prod_{n\leq t}\left( b \left(1 + \tpd w\right) - a_n q^{- B_{ij}} \left( 1 + \tpd{z_n}\right)\right) \nn.\eea
Now, given commuting indeterminates $F_n$, $1\leq n\leq s$, and $G$, define for each $\pi\in \Sigma_s$ and each $r\in\{1,\dots,s\}$,
\bea \nn A_{\pi,r}(F_1,\dots,F_s;G) &:=& 
\prod_{\substack{n<m\\ \pi^{-1}(n) < \pi^{-1}(m)}} 
\left(F_m - q^{- B_{ii}} F_n \right)
\prod_{\substack{n<m\\ \pi^{-1}(n) > \pi^{-1}(m)}} 
\left(q^{- B_{ii}} F_m  - F_n  \right) \nn\\
&&\times 
\prod_{\substack{n\\\pi^{-1}(n)> r}}
\left( F_n - q^{- B_{ij}} G  \right) 
\prod_{\substack{n\\\pi^{-1}(n)\leq r}}
\left( q^{- B_{ij}} F_n  - G \right).\nn\eea
We then have the identity
\be\nn \sum_{\pi\in\Sigma_s}\sum_{r=0}^s (-1)^r {\binom s r}_{q_i} A_{\pi,r}(F_1,\dots,F_s;G) = 0,\ee
which can be verified by direct calculation with the aid of a computer algebra system case-by-case (i.e. for the Cartan matrices of $A_1\times A_1$, $A_2$, $B_2$, $C_2$ and $G_2$). It is actually a special instance of an identity due to Jing which holds for arbitrary symmetric generalized Cartan matrices, \cite{Jing}, cf. also \cite{DJ} and \cite{HernandezFusion}. 
One may check that (\ref{autser1}) is equal to 
\be 
\frac 1 \Delta \left(\sum_{\pi\in\Sigma_s}\sum_{r=0}^s (-1)^r {\binom s r}_{q_i} 
A_{\pi,r}\left( a_1 \left(1+\pd{z_1}\right), \dots, a_s\left(1+\pd{z_s}\right); b\left(1+\pd{w}\right) \right) \right) \scr M.\nn
\ee
It therefore vanishes, as required.\end{proof}

\begin{prop}\label{cyclicserresufficeprop}
Given the relations (\ref{EErel}), imposing the relation 
\be \sum_{\pi\in\Sigma_s}\sum_{r=0}^s(-1)^r
{\binom s r}_{q_i} 
\E^\pm_{i,a_{\pi(1)}}(z_{\pi(1)})\ldots
  \E^\pm_{i,a_{\pi(r)}}(z_{\pi(r)})  \E^\pm_{j,b}(w) \E^\pm_{i,a_{\pi(r+1)}}(z_{\pi(r+1)})\ldots \E^\pm_{i,a_{\pi(s)}}(z_{\pi(s)}) =0\label{autser2}\ee
is equivalent to imposing the relation
\be \sum_{r=0}^s(-1)^r
{\binom s r}_{q_i} 
\E^\pm_{i,bq^{ \mp B_{ij} \mp r B_{ii}}}(z_{r+1})\ldots
  \E^\pm_{i,bq^{\mp B_{ij} \mp (s-1) B_{ii}}}(z_s)  \E^\pm_{j,b}(w) \E^\pm_{i,bq^{\mp B_{ij}}}(z_1) \ldots \E^\pm_{i,bq^{\mp B_{ij}  \mp (r-1) B_{ii} }}(z_{r}) =0.\label{imposeser}\ee
\end{prop}
\begin{proof}
By the preceding proposition, it is enough to consider the case in which the sticking graph of $\left\{ (i,a_1), (i,a_2), \dots, (i,a_s), (j,b)\right\}$ is a directed cycle graph. That is, cf Lemma \ref{cyclem}, we can suppose, relabelling the $a$'s as necessary, that 
\be a_t=bq^{\mp B_{ij} \mp (t-1) B_{ii}}=bq^{ \pm B_{ij} \pm (s-t) B_{ii}}.\nn\ee 
Obviously (\ref{autser2}) is equivalent to 
\be \sum_{r=0}^s(-1)^r
{\binom s r}_{q_i} \sum_{\pi\in\Sigma_s}
\E^\pm_{i,a_{\pi(r+1)}}(z_{\pi(r+1)})\ldots
  \E^\pm_{i,a_{\pi(s)}}(z_{\pi(s)})  \E^\pm_{j,b}(w) \E^\pm_{i,a_{\pi(1)}}(z_{\pi(1)})\ldots \E^\pm_{i,a_{\pi(r)}}(z_{\pi(r)}) =0.\nn\ee
The left-hand side here is equal to (by Proposition \ref{reEprop})
\be \sum_{r=0}^s(-1)^r
{\binom s r}_{q_i} \left(\sum_{\pi\in\Sigma_s} \mathscr D_\pi\right)
\E^\pm_{i,a_{r+1}}(z_{r+1})\ldots
  \E^\pm_{i,a_{s}}(z_{s})  \E^\pm_{j,b}(w) \E^\pm_{i,a_{1}}(z_{1})\ldots \E^\pm_{i,a_{r}}(z_{r})\nn\ee
where for each $\pi\in \Sigma_s$, $\mathscr D_\pi\in \C[[\pd {z_1},\dots \pd {z_s},\pd w]]$, and moreover  $\scr D_\pi$ has zero constant term for each $\pi\neq \id$, while $\scr D_\id= 1$. The result follows. 
\end{proof}


\subsection{Proof of Theorem \ref{mainthm}}\label{mainproof}

\begin{lem} \label{dellem}Let $u$ be an indeterminate. For all $n\in \Z_{\geq 0}$. 
\be u \left[\frac{a^{n+1}}{(n+1)!}\left(\pd a\right)^{n+1} \delta\left(\frac{a}{u}\right) \right]
=   a\left[\frac{a^{n+1}}{(n+1)!} \left(\pd a\right)^{n+1} \delta\left(\frac{a}{u}\right)   
+  \frac{a^n}{n!}\left(\pd a\right)^{n} \delta\left(\frac{a}{u}\right) \right] \nn \ee
\end{lem}
\begin{proof}
We have
\begin{align} u \left(\pd a\right)^{n+1} \delta\left(\frac{a}{u}\right) 
&= \left(\pd a\right)^{n+1} u\, \delta\left(\frac{a}{u}\right)  
= \left(\pd a\right)^{n+1} a\, \delta\left(\frac{a}{u}\right)   \nn\\
&= a \left(\pd a\right)^{n+1} \delta\left(\frac{a}{u}\right)  
+ (n+1) \left(\pd a\right)^{n} \delta\left(\frac{a}{u}\right)  \nn
\end{align}
and hence the result.
\end{proof}
\begin{prop}\label{thetadef}
The assignment
$k_i\mapsto \K_i$ and
\begin{align} x^\pm_{i}(z) &\mapsto \sum_{a\in \P_i} \sum_{m\in \Z_{\geq 0}} \E^\pm_{i,a,m} \frac{a^m}{m!}  \left(\pd a\right)^m \delta\left(\frac{a}{z}\right) ,\nn\\
\Phi_i(1/z) &\mapsto \sum_{a\in \P_i}\sum_{m\in \Z_{\geq 0}} \H_{i,a,m}\frac{a^m}{m!} \left(\pd a\right)^m \delta\left(\frac{a}{z}\right) \nn\end{align}
extends to a homomorphism of algebras $\theta_{\P}:\uqlg\to \alg$.
\end{prop}
\begin{proof}
In fact the defining relations of $\alg$ are constructed by demanding that this be true. Consider the relation (\ref{x+x+1}). On applying $\theta_\P$ to the left-hand side we have
\be \sum_{\substack{a\in \P_i,\\ m\in \Z_{\geq 0}}} \sum_{\substack{b\in \P_j ,\\ n\in \Z_{\geq 0}}} 
\left(\frac{a^m}{m!} \left(\pd a\right)^m \delta\left(\frac{a}{u}\right)\right)
\left(\frac{b^n}{n!} \left(\pd b\right)^n \delta\left(\frac{b}{v}\right) \right)
(u-q^{\pm B_{ij}}v)\E^\pm_{i,a,m}\E^\pm_{j,b,n}.\nn\ee
By Lemma \ref{dellem}, and the identity $\delta(u/a)u = \delta(u/a)a$, this is equal to 
\be\begin{split} \sum_{\substack{a\in \P_i,\\ m\in \Z_{\geq 0}}} \sum_{\substack{b\in \P_j,\\ n\in \Z_{\geq 0}}} 
\left(\frac{a^m}{m!} \left(\pd a\right)^m \delta\left(\frac{a}{u}\right)\right)
\left(\frac{b^n}{n!} \left(\pd b\right)^n \delta\left(\frac{b}{v}\right) \right)
\\\times\left[(a-bq^{\pm B_{ij}})\E^\pm_{i,a,m}\E^\pm_{j,b,n}
 +  a\E^\pm_{i,a,m+1}\E^\pm_{j,b,n} - bq^{\pm B_{ij}} \E^\pm_{i,a,m}\E^\pm_{j,b,n+1}\right]
\end{split}\label{eerelcheck}\ee
and we see that for the images $\theta_\P(x^\pm_{i,r})\in\alg$ to obey the relation (\ref{x+x+1}) it suffices to impose (\ref{EErel}). The relation (\ref{Phirel}) works in the same way. 

On applying $\theta_\P$ to (\ref{xpxmPhirel}) we find that the left-hand side is
\be \sum_{\substack{a\in \P_i,\\ m\in \Z_{\geq 0}}} 
\sum_{\substack{b\in \P_j,\\ n\in \Z_{\geq 0}}} 
\left(\frac{a^m}{m!} \left(\pd a\right)^m \delta\left(\frac{a}{u}\right)\right)
\left(\frac{b^n}{n!} \left(\pd b\right)^n \delta\left(\frac{b}{v}\right) \right) 
\left[ \E^+_{i,a,m}, \E^-_{j,b,n}\right]\nn\ee
while the right-hand side is
\bea &&\delta_{ij} \delta\left(\frac u v\right)  
\sum_{\substack{a\in \P_i,\\ M\in \Z_{\geq 0}}} 
\left(\frac{a^M}{M!} \left(\pd a\right)^M \delta\left(\frac{a}{u}\right)\right) 
\H_{i,a,M}\nn\\
&=& \delta_{ij}   
\sum_{\substack{a\in \P_i,\\ M\in \Z_{\geq 0}}} 
\left(\frac{a^M}{M!} \left(\pd a\right)^M \delta\left(\frac{a}{u}\right)
\delta\left(\frac{a}{v}\right)
\right) 
\H_{i,a,M}\nn\\
&=& \delta_{ij}   
\sum_{\substack{a\in \P_i,\\ m\in \Z_{\geq 0}}} 
\sum_{\substack{b\in \P_j,\\ n\in \Z_{\geq 0}}} 
\delta_{a,b}
\left(\frac{a^m}{m!} \left(\pd a\right)^m \delta\left(\frac{a}{u}\right)
\right) 
\left(\frac{b^n}{n!} \left(\pd b\right)^n \delta\left(\frac{b}{v}\right)
\right) 
\H_{i,a,m+n}\nn\eea
using the Leibniz rule. Therefore (\ref{EHrel}) is a sufficient condition for the images $\theta_\P(x^\pm_{i,r})$ and $\theta_\P(\Phi_{i,r})$ obey the relation (\ref{xpxmPhirel}). 

To ensure that the relation (\ref{Phidef3}) holds for the images $\K_i=\theta_\P(k_i)$ and $\theta_\P(\Phi_{i,r})$ it suffices to impose the linear relation (\ref{HKrel}). 

Finally, consider the Serre relations, (\ref{Serre}). The image under $\theta_\P$ of the left side of (\ref{Serre}) is 
\begin{align} 
 &\sum_{\pi\in\Sigma_s}\sum_{\substack{a_1,\dots,a_s\in \P_i,\, b\in \P_j\\ m_1,\dots,m_s,n\in \Z_{\geq 0}}} 
\left(\frac{b^{n}}{n!}  \left(\pd b\right)^{n} \delta\left(\frac{b}{z}\right)\right) \prod_{t=1}^s \left(\frac{a_t^{m_t}}{m_t!}  \left(\pd{a_t}\right)^{m_t} \delta\left(\frac{a_t}{w_{\pi(t)}}\right)\right) \nn\\ &\times\sum_{r=0}^s(-1)^r
{\binom s r}_{q_i} 
\E^\pm_{i,a_1,m_{1}}\ldots
  \E^\pm_{i,a_r,m_{r}}  \E^\pm_{j,b,n} \E^\pm_{i,a_{r+1},m_{r+1}}\ldots \E^\pm_{i,a_{s},m_{s}} ,\nn\\
=&\sum_{\pi\in\Sigma_s}\sum_{\substack{a_1,\dots,a_s\in \P_i,\, b\in \P_j\\ m_1,\dots,m_s,n\in \Z_{\geq 0}}} 
\left(\frac{b^{n}}{n!}  \left(\pd b\right)^{n} \delta\left(\frac{b}{z}\right)\right) \prod_{t=1}^s \left(\frac{a_{\pi(t)}^{m_{\pi(t)}}}{m_{\pi(t)}!}  \left(\pd{a_{\pi(t)}}\right)^{m_{\pi(t)}} \delta\left(\frac{a_{\pi(t)}}{w_t}\right)\right) \nn\\ &\times\sum_{r=0}^s(-1)^r
{\binom s r}_{q_i} 
\E^\pm_{i,a_{\pi(1)},m_{\pi(1)}}\ldots
  \E^\pm_{i,a_{\pi(r)},m_{\pi(r)}}  \E^\pm_{j,b,n} \E^\pm_{i,a_{\pi(r+1)},m_{\pi(r+1)}}\ldots \E^\pm_{i,a_{\pi s},m_{\pi s}} ,\nn\end{align} 
and for this to vanish it is sufficient to impose the relations (\ref{autser2}). But then, as Proposition \ref{cyclicserresufficeprop} states, given (\ref{EErel}), the relations (\ref{autser2}) are equivalent to  (\ref{cycleser}).
\end{proof}

It remains to prove the ``moreover'' part of Theorem \ref{mainthm}. We shall need the following lemma. 
\begin{lem}\label{finlem}
Suppose $V$ is a complex vector space. Let $F_{a,m}\in \End(V)$ be linear maps $V\to V$, where the label $a$ is drawn from a finite set of distinct points $\P \subset \Cx$ and where $m\in \{0,1,\dots M\}$ for some fixed $M\in \Z_{\geq 0}$. Let $z$ be an indeterminate. Suppose $\sum_{a\in \P}\sum_{m=0}^M \left(\frac{a^m}{m!} \left(\pd a\right)^m \delta\left(\frac a z\right) \right) F_{a,m}=0$. Then $F_{a,m}=0$ for all $a\in \P$ and $m\in \Z_{\geq 0}$.

More generally, let $F_{a_1,m_1;a_2,m_2;\dots; a_s,m_s}\in \End(V)$ and let $z_1,\dots, z_s$ be indeterminates. If $$\sum_{a_1,a_2,\dots, a_s\in \P}\sum_{m_1,m_2,\dots,m_s=0}^M F_{a_1,m_1;a_2,m_2;\dots; a_s,m_s} \prod_{t=1}^s\left(\frac{a_t^{m_t}}{m_t!} \left(\pd{a_t}\right)^{m_t} \delta\left(\frac{a_t}{z_t}\right) \right)=0$$ then the $F_{a_1,m_1;a_2,m_2;\dots; a_s,m_s}$ are all zero.
\end{lem}
\begin{proof} 
Let $$X(a) := \left( a^{n} \binom{n+m-2}{m-1} \right)_{\substack{1\leq n\leq M|\P|\\ 1\leq m\leq M}}.$$ 
We shall establish the identity
\be \det\bmx X(a_1)& X(a_2)& \dots &X(a_{|\P|}) \emx
   = \left(a_1a_2\dots a_{|\P|}\right)^{\frac{M(M+1)} 2} 
     \left( \prod_{i<j} (a_j-a_i) \right)^{M^2}.\label{detid} \ee  
The first part of the lemma follows from the fact that this determinant is not zero (consider equating coefficients of $z^{k}$ for $1\leq k\leq M|\P|$ in the equation given, and letting $\P = \{a_i^{-1}\}_{1\leq i \leq |\P|}$). 
The identity (\ref{detid})  follows by symmetry arguments analogous to those used in proving the Vandermonde determinant formula. 
Namely, let $D = \det\bmx X(a_1)& X(a_2)& \dots &X(a_{|\P|}) \emx$. This must be equal to a polynomial in the $a_i$ of order $1+2+\dots+M|\P| = \frac{M|\P|(M|\P|+1)} 2$. On symmetry grounds it must have a zero of order $M^2$ at $a_i=a_j$, for every pair $i\neq j$, and a zero of order $M+\binom M 2 =\frac{M(M+1)}2$ at $a_i=0$ for each $i$. Since $\frac{M|\P|(M|\P|+1)} 2 = M^2 \frac{|\P| (|\P|-1)} 2 + |\P| \frac{M(M+1)}2$, these are all the zeros of $D$, and (\ref{detid}) must hold up to a constant of proportionality. To see that this constant is unity, consider the coefficient of $a_1^{1+2+\dots +M} a_2^{(M+1) + \dots + 2M} \dots a_{|\P|}^{((|\P|-1)M+1) + \dots + |\P|M}$. This term comes only from the block on-diagonal part of the determinant sum. Hence the coefficient is 1 by virtue of the identity 
\be \det\left( \binom{n+m+k-2}{m-1} \right)_{\substack{1\leq n,m\leq M}} = 1, \label{detbin}\ee
valid for all non-negative integers $k$ (actually for all $k\in \C$; see for example \cite{ChudiClaudio}, Proposition 3.6, of which (\ref{detbin}) is a special instance).  

The ``more generally'' part follows by applying the first part $s$ times.
\end{proof}

We can now complete the proof of Theorem \ref{mainthm}.
\begin{prop}\label{pullbackprop}
Let $V\in \Ob(\cc_{\P})$. Then $V$ is the pull-back via $\theta_\P$ of a finite-dimensional representation of $\alg$.
\end{prop}
\begin{proof}
Proposition \ref{idea} guarantees that the actions of the generators of $x_{i}^\pm(z)$ and $\Phi_i(z)$ on $V$ are of the form (\ref{Emap}--\ref{Hmap}) for some linear maps $E^\pm_{i,a,m}$ and $H_{i,a,m}$. It is enough to show that these maps (together with the representatives of the $k_i^{\pm 1}$) obey the defining relations of $\alg$. 

Consider the relation (\ref{EErel}). Since the representatives of $x_i^\pm(z)$ by assumption obey relation (\ref{x+x+1}), we have
\be\begin{split} 0= \sum_{a,b\in \P_i} \sum_{m,n=0}^M 
\left(\frac{a^m}{m!} \left(\pd a\right)^m \delta\left(\frac{a}{u}\right)\right)
\left(\frac{b^n}{n!} \left(\pd b\right)^n \delta\left(\frac{b}{v}\right) \right)
\\\times\bigg[ (a-bq^{\pm B_{ij}})E^\pm_{i,a,m}E^\pm_{j,b,n}
 +  aE^\pm_{i,a,m+1}E^\pm_{j,b,n} - bq^{\pm B_{ij}} E^\pm_{i,a,m}E^\pm_{j,b,n+1}\\
- (aq^{\pm B_{ij}} -b)E^\pm_{i,a,m}E^\pm_{j,b,n}
   + aq^{\pm B_{ij}}E^\pm_{i,a,m+1}E^\pm_{j,b,n} - bE^\pm_{i,a,m}E^\pm_{j,b,n+1}\bigg]
\end{split},\nn\ee
where now the sum is over finitely many values of $m,n$, in contrast to (\ref{eerelcheck}). Therefore Lemma \ref{finlem} applies. Consequently the maps $E^\pm_{i,a,m}$ must indeed  satisfy (\ref{EErel}). The remaining relations work in the same way: for brevity we omit the details, which amount to introducing upper limits $M$ on the sums in the proof of Proposition \ref{thetadef} and replacing each ``sufficient'' statement with the corresponding ``necessary'' statement.
\end{proof}

\begin{rem}\label{truncrem} In fact we have that, given any $V\in \Ob(\cc_\P)$, there is an $M$ (the $M$ of Proposition \ref{idea}) such that $V$ is the pull-back not merely of a representation of $\alg$, but of the ``truncated'' algebra $\alg/\mc I_{M+1}$. And the proof above uses the fact that, for any given $M$, the homomorphism $\uqlg\to \alg/\mc I_{M+1}$ to the truncated algebra has a right-inverse, as the identity (\ref{detid}) ensures. (This right-inverse is far from unique. For each $i\in I$, its codomain is spanned by the modes numbered $-1,-2,\dots,-M|\P_i|$. But this was a choice: any $M|\P_i|$ distinct modes would have done.)  
However, 
the inverse of the matrix in (\ref{detid}) is not stable under changes in $M$, in the sense that the inverse of the restriction is not the restriction of the inverse, so we cannot take the inverse limit $M\to \8$ of these matrix inverses, c.f. \S\ref{completion}. (Nor is it stable under the addition of more points to the set $\P$.) 
\end{rem}

\section{Properties and first applications of $\alg$}\label{appssec}
\subsection{Triangular decomposition: weak form}
Let $\alg^\pm$ be the subalgebra of $\alg$ generated by $(\E^\pm_{i,a,m})_{i\in I,a\in \P_i,m\in \Z_{\geq 0}}$, and $\alg^0$ the subalgebra generated by $(\H_{i,a,m})_{i \in I, a\in \P_i,m\in \Z_{\geq 0}}$ and $(\K^{\pm 1}_i)_{i\in I}$. 
\begin{prop}\label{triangprop} $\alg = \alg^- \cdot \alg^0 \cdot \alg^+$.
\end{prop}
\begin{proof}
Given that $\E^+_{i,a,m} \E^-_{j,b,n} \in \alg^- \cdot \alg^+ + \alg^0$ by relation (\ref{EpEmrel}), it suffices to check that \be\nn \E^+_{i,a,m} \H_{j,b,n}\in \alg^0\cdot \alg^+\quad\text{ and }\quad \H_{i,a,m}\E^-_{j,b,n} = \alg^-\cdot\alg^0.\ee Consider the second of these (the first is similar).   
Whenever $a\neq bq^{-B_{ij}}$, it follows from Proposition \ref{reEprop}.
It remains to consider $\H_{i,a,m} \E^-_{\smash{j,aq^{B_{ij}},n}}$. Whenever $m>0$, (\ref{EHrel}) can be used, in the form
\begin{align} \H_{i,a,m} \E^-_{\smash{j,aq^{B_{ij}},n}} &=  \H_{i,a,m-1} \E^-_{\smash{ j,aq^{B_{ij}},n+1}} \nn\\\nn &+ (q^{- B_{ij}}-q^{B_{ij}}) \E^-_{\smash{j,aq^{B_{ij}},n}} \H_{i,a,m-1}  + q^{- B_{ij}} \E^-_{\smash{j,aq^{B_{ij}},n}} \H_{i,a,m}  - q^{B_{ij}} \E^-_{\smash{j,aq^{B_{ij}},n+1}} \H_{i,a,m-1} .\end{align}
Note that here we are ``solving downwards'', in the sense that  $\H_{i,a,m} \E^-_{\smash{j,aq^{B_{ij}},n}}$ is in $\mc I_{\max(m,n)}$ and yet we are re-writing it in a form that is not manifestly so (if $m\leq n$). Here it is useful to do so because by recursive application of this relation one has 
\be \H_{i,a,m} \E^-_{\smash{j,aq^{B_{ij}},n}} - \H_{i,a,0} \E^-_{\smash{j,aq^{B_{ij}},n+m}} \in \alg^-\cdot \alg^0,\nn\ee
and it remains only to consider $\H_{i,a,0} \E^-_{\smash{j,aq^{B_{ij}},n}}$, $n\in \Z_{\geq 0}$. For such terms, (\ref{EHrel}) yields no relation. 
Instead, one substitutes for $\H_{i,a,0}$ using relation (\ref{HKrel}):
\be \H_{i,a,0} \E^-_{\smash{j,aq^{B_{ij}},n}} = \left(\frac{\K_i-\K_i^{-1}}{q_i-q_i^{-1}} - \sum_{b\in\P_i\setminus\{a\}} \H_{i,b,0}\right) \E^-_{\smash{j,aq^{B_{ij}},n}}.\nn\ee
We have $\sum_{b\in\P_i\setminus\{a\}} \H_{i,b,0} \E^-_{\smash{j,aq^{B_{ij}},n}} \in \alg^-\cdot\alg^0$ by the arguments above; and  $\frac{\K_i-\K_i^{-1}}{q_i-q_i^{-1}} \E^-_{\smash{j,aq^{B_{ij}},n}}\in \alg^-\cdot \alg^0$ by  (\ref{Krels}).  
\end{proof}
However, since the tuple $\P=(\P_i)_{i \in I}$ consists of finite sets, we also have the following.
\begin{prop}\label{nottriang}
$\alg\not\cong_\C \alg^-\otimes \alg^0\otimes \alg^+$.
\end{prop}
\begin{proof} Pick an $i\in I$. We have
\begin{align} \K_i\E^+_{i,a,m} = q_i^2 \E^+_{i,a,m} \K_i &= q_i^2 \E^+_{i,a,m} \K_i^{-1} + q_i^2(q_i-q_i^{-1}) \E^+_{i,a,m} \sum_{b\in \P_i} \H_{i,b,0} \nn\\
&= q_i^4 \K_i^{-1} \E^+_{i,a,m} + q_i^2(q_i-q_i^{-1}) \E^+_{i,a,m} \sum_{b\in \P_i} \H_{i,b,0} \nn\\
&= q_i^4 \K_i \E^+_{i,a,m} - q_i^4(q_i-q_i^{-1}) \sum_{b\in \P_i} \H_{i,b,0} \E^+_{i,a,m} 
+ q_i^2(q_i-q_i^{-1}) \E^+_{i,a,m} \sum_{b\in \P_i} \H_{i,b,0}\nn  \end{align}
Since $\P_i$ is finite, there exists an $a\in \P_i$ such that $aq_i^{-2}\notin \P_i$. Pick any such $a$. Then Proposition \ref{reEprop} allows $\E^+_{i,a,m} \H_{i,b,0}$ to be reexpressed as a sum of monomials $\H_{i,b,n} \E^+_{i,a,m+k}$ with $n,k\in \Z_{\geq 0}$. On doing so, and re-arranging, one has the relation 
\be 0 = \frac{1}{q_i - q_i^{-1}} \K_i \E_{i,a,m}^+
    + \sum_{b\in \P_i} \sum_{n,k=0}^\8  \H_{i,b,n} \E^+_{i,a,m+k} \frac{a^k}{k!}\left(\pd a\right)^k \frac{b^n}{n!}\left(\pd b\right)^n \frac{b}{aq_i^{-2} - b}.\label{nottrel}\ee
\end{proof}
\begin{rem}
Another way to derive (\ref{nottrel}) is to consider the image of the relation \be\label{relll}x^+_i(z) \phi^+_i(1/w) =  \, q_i^{-2} \frac{1 - q_i^2 z/w}{1 - q_i^{-2}z/w} \, \phi^+_i(1/w) x^+_{i}(z)\ee under the homomorphism $\theta_\P$. We have \be\nn\theta_\P(\phi^+_i(1/w)) = \K_i + (q_i-q_i^{-1}) \sum_{b\in \P_i} \sum_{n=0}^\8 \H_{i,b,n}  \frac{b^n}{n!} \left(\pd b\right)^n \frac{b/w}{1-b/w}.\ee Comparing the formal partial fraction decompositions of the left- and right-hand sides of (\ref{relll}), one recovers  (\ref{nottrel}) by considering the coefficients of the poles at $w= aq_i^{-2}$ (which occur only on the right-hand side, provided $aq_i^{-2}\notin \P_i$). 
\end{rem}

\subsection{Results and examples in rank 1}\label{sl2sec} For this subsection, set $\g=\mf{sl}_2$. The Dynkin  diagram $A_1$ has one node only and we omit its label $i\in I=\{1\}$ throughout. 
We consider the case $\P = \{a,aq^2,aq^4,\dots, aq^K\}$ and write, by abuse of notation,
\be \E^+_{k,m} := \E^+_{aq^k,m},\quad \E^-_{k,m} := \E^-_{aq^k,m},\quad \H_{k,m} := \H_{aq^k,m},\quad 
Y_k := Y_{aq^k},\quad A_k := A_{aq^k}. \nn\ee 
We have the principal gradation of $\alg$,
\be \alg^\pm = \bigoplus_{r\in \Z_{\geq 1}} \alg^{(\pm r)}, \quad\text{where}\quad 
\alg^{(\pm r)} := \left\{ x\in \alg^\pm : \K x\K^{-1} = q^{\pm 2r} x \right\}. \nn\ee

Let $B_r$ be the set of all monomials of the form
$\E^-_{k_1,m_1} \E^-_{k_2,m_2} \dots \E^-_{k_r,m_r}$
such that, for each $t\in \{1,2,\dots,r-1\}$, $k_t\leq k_{t-1}$ and if $k_t=k_{t+1}$ then $m_t < m_{t+1}$.  
\begin{prop} $B_r$ is a $\C$-basis of $\alg^{(-r)}$.\label{sl2basisprop} 
\end{prop}
\begin{proof} 
$\alg^{(-r)}$ is spanned by the monomials of the form $\E^-_{k_1,m_1} \E^-_{k_2,m_2} \dots \E^-_{k_r,m_r}$. By re-writing neighbouring factors according to the rules below, which follow from the defining relation (\ref{EErel}), any such monomial can be expressed as a linear combination of elements of $B_r$.
\begin{enumerate}[Rule i):]
\item\be\E^-_{k,m} \E^-_{k,m}\,\, \underset\mapsto= \,\,\frac{1}{1-q^{-2}}
                \bigl( - \E^-_{k,m+1} \E^-_{k,m} + q^{-2} \E^-_{k,m} \E^-_{k,m+1}\bigr).\nn\ee
\item If $k>l$, of if $k=l$ and $m>n$, then
\begin{align}
\E^-_{k,m} \E^-_{\ell,n}\,\, &\underset\mapsto= \,\,\frac{1}{q^k-q^{\ell-2}}
                \bigl( -q^k \E^-_{k,m+1} \E^-_{\ell,n} + q^{\ell-2} \E^-_{k,m} \E^-_{\ell,n+1}\nn\\
       &  \qquad\qquad\qquad +q^{k-2} \E^-_{\ell,n} \E^-_{k,m+1} - q^\ell \E^-_{\ell,n+1}\E^-_{k,m} 
                      + (q^{k-2} - q^\ell) \E^-_{\ell,n} \E^-_{k,m} \bigr).\nn
\end{align}
\end{enumerate} 
For elements of $B_r$ no further re-writing is possible, and it is moreover clear that the defining relation (\ref{EErel}) does not give any linear relations between elements of $B_r$.
\end{proof}

Suppose now that $\bs\gamma=(\gamma(u))$ is a dominant $\ell$-weight, cf. \S\ref{lweightssec}. The irreducible $\L(\bs\gamma)$ is isomorphic to a tensor product of evaluation modules \cite{CPsl2}. This provides an explicit basis for $\L(\bs\gamma)$, but one which is not generally compatible with its decomposition into $\ell$-weight spaces. In cases when all $\ell$-weight spaces have dimension one (the module $\L(\bs\gamma)$ is then called thin/quasi-minuscule, and in type $A_1$ this happens precisely when all poles of $\gamma(u)$ are simple) explicit $\ell$-weight bases can be found in \cite{Thoren,YoungZegers}. However, when $\gamma(u)$ has poles of higher order, it is a nontrivial task to find $\ell$-weight bases of $\L(\bs\gamma)$ and thence, in particular, to determine the Jordan block structure of the generators $\phi^\pm_{\pm r}$. 
\label{basessec}

Figure \ref{basesfig} shows three examples illustrating the use of the algebra $\alg$ in constructing such $\ell$-weight bases. In each case, the graph of the $q$-character is shown together with a basis of the representation consisting of vectors in $B\on v$, where
\be B:= \bigcup_r B_r \nn\ee 
is a basis of $\alg^-$ by the above proposition, and where $v$ is defined to be a simultaneous eigenvector of $\K$ and $\H_{k,m}$ such that $\E^+_{k,m}\on v=0$ for all $k,m$. The eigenvalues $\lambda$ of $\K$ and $\lambda_{k,m}$ of $\H_{k,m}$ are read off from the partial fraction decomposition of $\gamma(u)$ according to  -- cf.  \S\ref{lweightssec} and (\ref{Hmap}) --
\be \gamma(u) = \lambda + \sum_{k,m}  \lambda_{k,m} \frac{a^m}{m!} \left(\pd a\right)^m \frac{aq^ku}{1-aq^k u}. \label{pfrac}\ee

In each case, the action of the generators $\H_{k,m}$ and $\E^+_{k,m}$ in this basis can be computed 
using the defining relations and their known action on the highest weight vector $v$. For example, in the module $L(Y_0^2 Y_2)$, one has
\be \E^+_{0,0} \on (\E^-_{0,1} \E^-_{2,0}\on v) = \H_{0,1} \E^-_{2,0} \on v
                                                = -(q^{2} - q^{-2}) \E^-_{2,0} \H_{0,0} \on v
                                                = (q^2-q^{-2}) \E^-_{2,0} \on v.\nn\ee

The action of the generators $\E^-_{k,m}$ is found by direct calculation, working in the induced $\alg$-module $\alg \otimes_{\alg^0\otimes \alg^+} \C v$ whose irreducible quotient is $\L(\bs\gamma)$. 
In the first two examples in Figure \ref{basesfig}, namely $L(Y_0Y_2^2)$ and $L(Y_0^2Y_2)$, one finds that the set of all vectors in $B\on v$ that are not zero in $L(\bs\gamma)$ form a basis of $\L(\bs\gamma)$. It is natural to ask whether this property might be true in general (cf. Lusztig's canonical bases \cite{Lusztig} -- see e.g. Theorem 14.2.5 in \cite{CPbook} or Theorem 11.16 in \cite{Jantzen}). 
The third example, $L(Y_0^3Y_0^2)$, shows, however, that this is not so.
The $\ell$-weight space of $\ell$-weight $Y_0^2Y_4^{-1}$ has dimension 4, but one finds that 5 vectors of $B.v\subset \alg \otimes_{\alg^0\otimes \alg^+} \C v$  are not singular, namely
\be   \E^-_{0,0} \E^-_{2,0} \on v,\quad  \E^-_{0,0} \E^-_{2,1} \on v,\quad 
\E^-_{0,1} \E^-_{2,0} \on v ,\quad  \E^-_{0,1} \E^-_{2,1} \on v,\quad  \E^-_{0,2} \E^-_{2,0} \on v .\nn\ee
Some linear combination must therefore be singular. And indeed, 
$\E^+_{0,0} \E^-_{0,2} \E^-_{2,0} \on v = \H_{0,2} \E^-_{2,0} \on v = \H_{0,1} \E^-_{2,1} \on v$, while also 
$\E^+_{0,0} \E^-_{0,1} \E^-_{2,1} \on v = \H_{0,1} \E^-_{2,1} \on v$. In this way one checks  that $\left(\E^-_{0,2} \E^-_{2,0}- \E^-_{0,1} \E^-_{2,1} \right) \on v\notin B\on v$ is a singular vector. Thus $\E^-_{0,2}\on (\E^-_{2,0} \on v) = \E^-_{0,1} \E^-_{2,1} \on v$ in $L(Y_0^3 Y_2^2)$. 


\begin{figure}
\begin{itemize} 
\item[$L(Y_0Y_2^2)$:]  $\K\on v= q^3 v,\quad  \H_{2,0} \on v = (q^2 + 1+q^{-2})v,\quad \H_{2,1}\on v= (q^2-q^{-2})v $.
\be
\begin{tikzpicture}    
\matrix (m) [matrix of math nodes, row sep=2em, column sep=1em, text height=1.5ex, text depth=1ex]    
{  Y_0 Y^2_2  &     & \\    
              &   2Y_0 Y_2 Y_4^{-1}  &   \\
                   &   Y_4^{-1}      &  Y_0 Y_4^{-2}   \\
                   &   &   Y_2^{-1}Y_4^{-2} \\ };   
\path[->,thick,font=\scriptsize]    
(m-1-1) edge node[fill=white,inner sep=1pt] {$ A_{2}^{-1}  $} (m-2-2)    
(m-2-2) edge node[fill=white,inner sep=1pt] {$ A_{2}^{-1}  $} (m-3-3)    
(m-3-2) edge node[fill=white,inner sep=1pt] {$ A_{2}^{-1}  $} (m-4-3)    
(m-2-2) edge node[fill=white,inner sep=1pt] {$ A_{0}^{-1}  $} (m-3-2)    
(m-3-3) edge node[fill=white,inner sep=1pt] {$ A_{0}^{-1}  $} (m-4-3)    
;
\end{tikzpicture}\begin{tikzpicture}    
\matrix (m) [matrix of math nodes, row sep=2em, column sep=1em, text height=1.5ex, text depth=1ex]    
{  v  &     & \\    
              &  \E^-_{2,0}\on v,\E^-_{2,1}\on v  &   \\
                   &  \E^-_{0,0} \E^-_{2,0}\on v   & \E^-_{2,0} \E^-_{2,1}\on v  \\
                   &   &  \E^-_{0,0} \E^-_{2,0} \E^-_{2,1}\on v   \\ };   
\path[->,thick,font=\scriptsize]    
(m-1-1) edge (m-2-2)    
(m-2-2) edge (m-3-3)    
(m-3-2) edge (m-4-3)    
(m-2-2) edge (m-3-2)    
(m-3-3) edge (m-4-3)    
;
\end{tikzpicture}
\nn\ee
\item[$L(Y_0^2Y_2)$:]  $\K\on v= q^3 v,\quad \H_{0,0} \on v = -v,\quad \H_{2,0}\on v= (q^2+2+q^{-2})v$.
\be
\begin{tikzpicture}    
\matrix (m) [matrix of math nodes, row sep=2em, column sep=1em, text height=1.5ex, text depth=1ex]    
{  Y_0^2 Y_2  &      \\    
  Y_0      &   Y_0^2Y_4^{-1}   \\
                   &   2Y_0 Y_2^{-1} Y_4^{-1}    \\
                   &    Y_2^{-2}Y_4^{-1}   \\    };   
\path[->,thick,font=\scriptsize]    
(m-1-1) edge node[fill=white,inner sep=1pt] {$ A_{2}^{-1}  $} (m-2-2)    
(m-2-1) edge node[fill=white,inner sep=1pt] {$ A_{2}^{-1}  $} (m-3-2)    
(m-1-1) edge node[fill=white,inner sep=1pt] {$ A_{0}^{-1}  $} (m-2-1)    
(m-2-2) edge node[fill=white,inner sep=1pt] {$ A_{0}^{-1}  $} (m-3-2)    
(m-3-2) edge node[fill=white,inner sep=1pt] {$ A_{0}^{-1}  $} (m-4-2)    
;
\end{tikzpicture}
\begin{tikzpicture}    
\matrix (m) [matrix of math nodes, row sep=2em, column sep=1em, text height=1.5ex, text depth=1ex]    
{  v  &     & \\    
  \E^-_{0,0} \on v      &   \E^-_{2,0}\on v   \\
                   &   \E^-_{0,0} \E^-_{2,0} \on v,\E^-_{0,1} \E^-_{2,0} \on v    \\
                   &    \E^-_{0,0}\E^-_{0,1} \E^-_{2,0} \on v   \\   };   
\path[->,thick,font=\scriptsize]    
(m-1-1) edge  (m-2-2)    
(m-2-1) edge  (m-3-2)    
(m-1-1) edge  (m-2-1)    
(m-2-2) edge  (m-3-2)    
(m-3-2) edge  (m-4-2)    
;
\end{tikzpicture}
\nn\ee
\item[$L(Y_0^3Y_2^2)$:]  $\K\on v= q^5 v,\quad \H_{0,0} \on v = v,\quad \H_{2,0}\on v = (q^4+q^2+q^{-2} + q^{-4}) v,\quad \H_{2,1}\on v= (q^4+2q^2-2q^{-2}-q^{-4})v$.
\be
\begin{tikzpicture}    
\matrix (m) [font=\scriptsize,matrix of math nodes, row sep=2em, column sep=0em, text height=1.5ex, text depth=1ex]    
{  Y_0^3 Y_2^2  &     & \\    
  Y_0^2 Y_2     &   2Y_0^3Y_2Y_4^{-1}  &   \\
                   &   4Y_0^2 Y_4^{-1}     &  Y_0^3Y_4^{-2} \\
                   &   2Y_0Y_2^{-1}Y_4^{-1}&3Y_0^2Y_2^{-1}Y_4^{-2}\\    
                   &                          &3Y_0  Y_2^{-2}Y_4^{-2}\\
                   &                          & Y_2^{-3}Y_4^{-2}. \\ };   
\path[->,thick,font=\scriptsize]    
(m-1-1) edge node[fill=white,inner sep=1pt] {$ A_{2}^{-1}  $} (m-2-2)    
(m-2-1) edge node[fill=white,inner sep=1pt] {$ A_{2}^{-1}  $} (m-3-2)    
(m-2-2) edge node[fill=white,inner sep=1pt] {$ A_{2}^{-1}  $} (m-3-3)    
(m-3-2) edge node[fill=white,inner sep=1pt] {$ A_{2}^{-1}  $} (m-4-3)       
(m-4-2) edge node[fill=white,inner sep=1pt] {$ A_{2}^{-1}  $} (m-5-3)       
(m-1-1) edge node[fill=white,inner sep=1pt] {$ A_{0}^{-1}  $} (m-2-1)    
(m-2-2) edge node[fill=white,inner sep=1pt] {$ A_{0}^{-1}  $} (m-3-2)    
(m-3-2) edge node[fill=white,inner sep=1pt] {$ A_{0}^{-1}  $} (m-4-2)    
(m-3-3) edge node[fill=white,inner sep=1pt] {$ A_{0}^{-1}  $} (m-4-3)    
(m-4-3) edge node[fill=white,inner sep=1pt] {$ A_{0}^{-1}  $} (m-5-3)    
(m-5-3) edge node[fill=white,inner sep=1pt] {$ A_{0}^{-1}  $} (m-6-3)       
;
\end{tikzpicture}
\begin{tikzpicture}    
\matrix (m) [font=\scriptsize,matrix of math nodes, row sep=2em, column sep=-4em, text height=1.5ex, text depth=1ex]    
{  v  &     & \\    
  \E^-_{0,0}\on v     &   \E^-_{2,0}\on v, \E^-_{2,1}\on v  &   \\
           &   \E^-_{0,0} \E^-_{2,0} \on v,  \E^-_{0,0} \E^-_{2,1} \on v, 
  \E^-_{0,1} \E^-_{2,0} \on v ,  \E^-_{0,1} \E^-_{2,1} \on v  &  \E^-_{2,0}\E^-_{2,1}\on v \\
                 &   \E^-_{0,0}\E^-_{0,1}\E^-_{2,0} \on v,\E^-_{0,0}\E^-_{0,1}\E^-_{2,1} \on v 
     &\E^-_{0,0} \E^-_{2,0} \E^-_{2,1} \on v,\E^-_{0,1} \E^-_{2,0} \E^-_{2,1} \on v,\E^-_{0,2} \E^-_{2,0} \E^-_{2,1} \on v\\    
                   &       &\E^-_{0,0} \E^-_{0,2} \E^-_{2,0} \E^-_{2,1} \on v,\E^-_{0,1} \E^-_{0,2} \E^-_{2,0} \E^-_{2,1} \on v,\E^-_{0,0} \E^-_{0,1} \E^-_{2,0} \E^-_{2,1} \on v\\
                   &                          & \E^-_{0,0} \E^-_{0,1} \E^-_{0,2}\E^-_{2,0} \E^-_{2,1} \on v. \\ };   
\path[->,thick,font=\scriptsize]    
(m-1-1) edge  (m-2-2)    
(m-2-1) edge (m-3-2)    
(m-2-2) edge  (m-3-3)    
(m-3-2) edge  (m-4-3)       
(m-4-2) edge  (m-5-3)       
(m-1-1) edge  (m-2-1)    
(m-2-2) edge  (m-3-2)    
(m-3-2) edge  (m-4-2)    
(m-3-3) edge  (m-4-3)    
(m-4-3) edge  (m-5-3)    
(m-5-3) edge  (m-6-3)       
;
\end{tikzpicture}
\nn\ee
\end{itemize}
\begin{caption}{Examples of $\ell$-weight bases of simple $U_q(\mc L \mf{sl}_2)$-modules. See \S\ref{basessec}.\label{basesfig}} 
\end{caption}
\end{figure}

\subsection{Comment on truncations and triangular decompositions.} 
Recall -- see Remark \ref{truncrem} -- that for every $V\in \Ob(\cc_\P)$ there is an $M$ such that $V$ is the pull-back of a representation of the truncated algebra $\alg/\mc I_{M+1}$. Here we consider how highest weight $\uqlg$-modules can fail to be pull-backs of highest weight $\alg/\mc I_N$-modules if $N$ is too small. 

Let $\bs\gamma=(\gamma_i(u))_{i\in I}$ be the highest $\ell$-weight. If $\gamma_i(u)$ has any pole of order $>N$ then the problem is obvious, cf. (\ref{pfrac}). For instance, the first example in Figure \ref{basesfig} is not a pull-back of a representation of $\alg/\mc I_1$.   
But there are more subtle possibilities, as the following pair of examples illustrate.

We set $\P = \{aq^k:0\leq k\leq K\}$ for some $a\in \Cx$ and $K\in \Z_{\geq 0}$. 
\begin{itemize}
\item Consider $\L(Y_{1,a} Y_{2,aq})$ in type $A_2$. All poles of the functions $(Y_{1,a}Y_{2,aq})_i(u)$, $i=1,2$, are simple. Yet $\L(Y_{1,a} Y_{2,aq})$ is not thin (indeed, $\chi_q(\L(Y_{1,a}Y_{2,aq})) = Y_{1,a} Y_{2,aq} + Y_{2,aq}^2 Y_{1,aq^2}^{-1} + 2Y_{2,aq} Y_{2,aq^3}^{-1} + \dots$) so it should not be a pull-back of a representation of $\alg/\mc I_1$. 
Suppose it were.
Then this representation would have to be highest weight with highest weight vector $v$ such that 
$\H_{1,a,0}\on v= v$, $\H_{2,aq,0}\on v=v$ and all others zero. Now, in $\alg/\mc I_1$, one has the relation $\E^-_{1,a,0} \H_{2,aq,0} = 0$.  Hence $\E^-_{1,a,0} \H_{2,aq,0}\on v=0$ and therefore $\E^-_{1,a,0}\on v=0$. But then $\H_{1,a,0}\on v = [\E^+_{1,a,0},\E^-_{1,a,0}] \on v = 0$,  which is a contradiction unless  $v=0$.   
\item Similarly, consider $\L(Y_{1,a}Y_{3,aq^2})$ in type $A_3$. It is generated by a highest weight vector $v$ such that $\H_{1,a,0} \on v= \H_{3,aq^2,0}\on v =v$. Hence $\E^-_{2,aq,0} \E^-_{1,a,0} \on v  = \E^-_{2,aq,0} \E^-_{1,a,0} \H_{3,aq^2,0}\on v = 
 \E^-_{2,aq,0} \H_{3,aq^2,0} \E^-_{1,a,0}\on v = 0$ since $\E^-_{2,aq,0} \H_{3,aq^2,0}=0$ in $\alg/\mc I_1$. But, once we compute $\H_{2,aq,0} \E^-_{1,a,0}\on v =   \E^-_{1,a,0} \on v$, we find that $\E^+_{1,a,0} \E^+_{2,aq,0}  \E^-_{2,aq,0} \E^-_{1,a,0}  \on v
   =  \E^+_{1,a,0} \E^-_{1,a,0} \on v = \H_{1,a,0} \on v = v$
which again is a contradiction unless $v=0$.
\end{itemize}
What underlies these examples is the fact that in the truncated algebras $\alg/\mc I_N$ one has relations like $\E^-_{i,aq^{B_{ii}},N-1} \H_{i,a,N-1} = 0$. This relation holds for all $a$ such that $\{a,aq^{B_{ii}}\}\subset \P_i$. (By contrast, in Proposition \ref{nottrel}, the failure of triangularity had to do with ``end'' points, i.e. to points $a\in \P_i$ such that for example $aq_i^{-2}\notin \P_i$.)

\subsection{A ``rational limit'' of $\alg$}
Just as Yangians are related to quantum loop algebras, \cite{GautamToledanoLaredo,GuayMa}, one might expect there to be a ``rational'' version of $\alg$, whose defining relations do not involve a deformation parameter. Here we merely note that a natural candidate can be defined.

For this subsection (only) let $q= e^h$ with $h$ an indeterminate. Pick finite sets $\overline \P_i\subset \C$ for each $i\in I$ and let $\P_i = \{q^k:k\in \overline \P_i\}$. We define new generators $\Er^\pm_{i,k,m}$, $\Hr^\pm_{i,k,m}$ and $\HHr_i$ by
\be \E^\pm_{i,k,m} := h^m \Er^\pm_{i,q^k,m} ,  \qquad
   \H^\pm_{i,k,m} := h^m \Hr^\pm_{i,q^k,m} , \qquad \K_i := e^{hr_i \HHr_i}.\nn\ee
Note that $\frac{\K_i-\K_i^{-1}}{q_i-q_i^{-1}} = \HHr_i + O(h)$. If one then keeps the leading order in $h$ of the defining relations in the definition, \S\ref{algdef}, of $\alg$, one obtains  the following definition.  
Let $\algr$ be the associative unital algebra over $\C$ generated by
\be \Er^\pm_{i,k,m} \quad \Hr_{i,k,m}, \quad \HHr_i, \quad i\in I,\, k\in \overline \P_i,\, m\in \Z_{\geq 0},\nn\ee
subject to the following relations for all $i,j\in I$, $k,\ell\in \overline \P_i$ and $m,n\in\Z_{\geq 0}$:
\be \left[\HHr_i, \Er^\pm_{j,k,m}\right] = \pm C_{ij} \Er^\pm_{j,k,m}, \nn\label{HHErrel}
\quad \left[ \HHr_i,\HHr_j\right] = 0 \ee\be  \left[ \HHr_i , \Hr_{j,k,m}\right] = 0,\quad \left[\Hr_{i,k,m} \Hr_{j,\ell,n} \right]= 0 \nn\ee
\be \left[ \Er^+_{i,k,m}, \Er^-_{j,\ell,n} \right] = \delta_{ij} \delta_{k\ell} \Hr_{i,k,m+n} \nn\ee
\bea
&&(k-\ell\mp B_{ij}) \Er^\pm_{i,k,m} \Er^\pm_{j,\ell,n} + \Er^\pm_{i,k,m+1} \Er^\pm_{j,\ell,n} - \Er^\pm_{i,k,m} \Er^\pm_{j,\ell,n+1} \nn\\
&=& (k\pm B_{ij}-\ell) \Er^\pm_{j,\ell,n} \Er^\pm_{i,k,m}  +  \Er^\pm_{j,\ell,n} \Er^\pm_{i,k,m+1}  - \Er^\pm_{j,\ell,n+1} \Er^\pm_{i,k,m} \nn\label{EErrel}
\eea
\bea
&&(k-\ell\mp B_{ij}) \Hr_{i,k,m} \Er^\pm_{j,\ell,n} + \Hr_{i,k,m+1} \Er^\pm_{j,\ell,n} -  \Hr_{i,k,m} \Er^\pm_{j,\ell,n+1} \nn\\
&=& (k\pm B_{ij}-\ell) \Er^\pm_{j,\ell,n} \Hr_{i,k,m}  + \Er^\pm_{j,\ell,n} \Hr_{i,k,m+1}  - \Er^\pm_{j,\ell,n+1} \Hr_{i,k,m} \nn\label{EHrrel}
\eea
\be \sum_{k\in \Z} \Hr_{i,k,0} = \HHr_i \nn\label{HKrrel},\ee
together with Serre relations 
\be\begin{split} \sum_{r=0}^{s}(-1)^r
{\binom s r} 
\Er^\pm_{i,k\mp B_{ij} \mp r B_{ii},m_{r+1}}\ldots
\Er^\pm_{i,k\mp B_{ij} \mp (s-1) B_{ii},m_s}  \Er^\pm_{j,k,n} \Er^\pm_{i,k\mp B_{ij},m_1} \ldots \Er^\pm_{i,k\mp B_{ij}  \mp (r-1) B_{ii} ,m_{r}} =0,\nn\label{ratcycleser}\end{split}\ee
where $s=1-C_{ij}$, for each $i\neq j$ such that $B_{ij}\neq 0$, for all $m_1,\dots,m_s,n\in \Z_{\geq 0}$, and for all $k\in \overline\P_j$ such that $\{k\mp B_{ij} \mp (t-1) B_{ii}: 1\leq t\leq s\}\subset \overline\P_i$.

\def\cprime{$'$}
\providecommand{\bysame}{\leavevmode\hbox to3em{\hrulefill}\thinspace}
\providecommand{\MR}{\relax\ifhmode\unskip\space\fi MR }
\providecommand{\MRhref}[2]{%
  \href{http://www.ams.org/mathscinet-getitem?mr=#1}{#2}
}
\providecommand{\href}[2]{#2}

\end{document}